\documentclass[11pt, reqno, twoside]{article}

\usepackage{fancyhdr}

\usepackage[T1]{fontenc} 

\usepackage{tabularx, longtable}

\usepackage{graphicx, color}
\usepackage{amsmath} 
\usepackage{amssymb}
\usepackage{amsthm}
\usepackage{amsxtra}
\usepackage{amscd}
\usepackage{bbm}
\usepackage{mathrsfs}
\usepackage{bm}
\usepackage[british]{babel}
\usepackage{braket}

\usepackage{latexsym}
\usepackage{amssymb}
\usepackage{amsmath}
\usepackage{dsfont} 
\usepackage{graphicx}
\usepackage{hyperref}
\newcommand{\id}{\mathds{1}}

\newcommand{\comment}[1]{}

\def\dist{{\rm dist}\,}
\def\e{\mathrm{e}}

\def\gs{\psi_{E^\eps_{s},s}}
\def\gsv{\vec{\psi}_{E^\eps_{s},s}}
\def\gsvs{\vec{\psi}_{E^\eps_{s'},s'}}
\newcommand\japx {\langle x \rangle}
\def\jx4{\langle x \rangle^{-4}}

\makeatletter
\def\Links{\tagsleft@true}\def\Rechts{\tagsleft@false}
\makeatother

\def\dist{\vec{\phi}_s}

\def\distt{\tilde{\vec{\phi}}_s}
\def\dists{\tilde{\vec{\phi}}_{s'}}
\def\Ml{M^{(\mathrm{l})}_{s}}

\def\M{M^{(\mathrm{g})}_{s}}
\def\Mlt{M^{(\frac{3}{2},\mathrm{l})}_{s}}
\def\Mltn{M^{(\frac{3}{2},\mathrm{l})}_{s_{0}}}

\def \jte{\langle \frac s \eps \rangle}
\def \jse{\langle \frac {s'}{ \eps} \rangle}
\def \jtse{\langle \frac{s-s'}{\eps} \rangle}

\def \H{\|_{H^2}}
\def \Hl{\|_{H^{2,l}}}
\def \Hls{\|_{H^{2,-\sigma}}}

\def \Bsl{$(\mathrm{B}_{\mathrm{l}})$}
\def \Bs{$(\mathrm{B}_{\mathrm{g}})$}

\numberwithin{equation}{section}

\def\R{\mathbb R}
\def\C{\mathbb C}

\def\i{{\mathrm i}}
\def\e{{\mathrm e}}

\def\epsilon{\varepsilon}
\def\eps{\varepsilon}
\def\V{V_{s}}

\numberwithin{equation}{section}
\numberwithin{figure}{section}

\newtheorem{thm}{\bfseries Theorem}[section]
\newtheorem{lem}[thm]{\bfseries Lemma}
\newtheorem{cor}[thm]{\bfseries Corollary}
\newtheorem{prop}[thm]{\bfseries Proposition}
\newtheorem*{cor*}{\bfseries Corollary}
\usepackage{cite}

\usepackage[nottoc,notlof,notlot]{tocbibind} 
\usepackage{psfrag}

\newcommand{\pagenumberstyle}{}
\newcommand{\markstyle}{\small \scshape \MakeLowercase}

\title{An Adiabatic Theorem for the Gross-Pitaevskii Equation}
\author{Zhou Gang\footnote{gzhou@caltech.edu, partly supported by NSF grant DMS-1308985 and DMS-1443225.}\ \ and Philip Grech\footnote{pgrech@ethz.ch, the main part of this work has been carried out at the Institute for Theoretical Physics, ETH Zurich, 8093 Zurich, Switzerland}}

\begin{document}

\maketitle
\centerline{$^\star$ Department of Mathematics, California Institute of Technology, Mail Code 253-37, Pasadena, U.S.A. 91106}
\centerline{$^\dagger$Department of Management, Technology and Economics, ETH Zurich, 8092, Zurich, Switzerland
}
\begin{abstract}
We prove an adiabatic theorem for the non-autonomous Gross-Pitaevskii equation in the case of a weak trap. More precisely, we assume that the external potential decays suitably at infinity and admits exactly one bound state. \color{black}
\end{abstract}

\fancypagestyle{plain}{%
\fancyhf{} 
\fancyfoot[C]{\pagenumberstyle \thepage} 
\renewcommand{\headrulewidth}{0pt}
}

\pagestyle{plain}

{
\tableofcontents
}

\pagestyle{fancy}
\fancyhead{} 
\fancyfoot{}
\fancyhead[CE]{\markstyle \leftmark} 
\fancyhead[CO]{\markstyle \rightmark} 
\fancyhead[LE,RO]{\pagenumberstyle \thepage}
\renewcommand{\headrulewidth}{0pt}

\renewcommand{\sectionmark}[1]{
\markright{\MakeUppercase \thesection.\ #1}
}

\fancypagestyle{plain}{%
\fancyhf{} 
\fancyfoot[C]{\pagenumberstyle \thepage} 
\renewcommand{\headrulewidth}{0pt}
}

\newpage
 \label{chap4}
 
\section{Introduction}

Quantum adiabatic theory has been initiated with the study of the non-autonomous Schr\"odinger equation 
\begin{align}
\i\eps\partial_{s}\psi_{s}=H_{s}\psi_{s} \label{linschr}
\end{align}  
in the limit $\eps \searrow 0$. Here, $H_{s}$ is a time-dependent self-adjoint Hamiltonian and the macroscopic time variable $s$ is assumed to take values in $[0,1]$. The first adiabatic theorem was discovered by Born and Fock \cite{BoFo} in 1928 who treated the case where $H_{s}$ has a simple eigenvalue which remains isolated from the rest of the spectrum at all times. 

Since that time a wide range of generalizations have been found. Some authors considered the case of isolated yet degenerate eigenvalues or isolated energy bands \cite{K2, ASY87, ASY93}. Others were concerned with the development of superadiabatic expansions which approximate the solution of (\ref{linschr}) with exponential accuracy in $\eps$ \cite{Nen92, JP1}; much like in the well-known Landau-Zener Formula \cite{Landau, Zener}. It has also been found that an adiabatic theory can even be given for non-isolated eigenvalues at the cost of having no information on the rate of convergence as $\eps \searrow 0$ \cite{AE99, Bo}. Many of these theorems have later been further generalized to non-self-adjoint Hamiltonians which arise for instance in non-equilibrium statistical mechanics \cite{joye, nenras, WKAstat, afgg2}. 

The present article takes a slightly different approach and studies a {\itshape nonlinear} example of a quantum adiabatic theorem. More precisely, we consider the non-autonomous Gross-Pitaevskii equation with a time-dependent potential $V_{s}=V_{s}(x)$,
\begin{align} \label{eq_GP}
 \i \varepsilon\partial_s \Psi_s = -\Delta\Psi_s+V_s\Psi_s+
b|\Psi_s|^2\Psi_s \, ,
\end{align}
where $b=\pm1$ (focusing resp. defocusing nonlinearity). Equation~(\ref{eq_GP}) constitutes an effective description for the dynamics of a Bose-Einstein condensate with one-particle wave function $\Psi_{s}$ in an external trap $V_{s}$. It can be rigorously derived from the underlying many-body Schr\"odinger dynamics in the limit of large particle numbers $N \to \infty$ if the interaction potential between the particles is scaled suitably with $N$ (see e.g. \cite{Sch09, Pi10} and references therein). However, it is worthwhile to note that such results are not uniform in the macroscopic time $t:=s/\eps$ and hence we will simply take Equation~(\ref{eq_GP}) as our starting point: Issues concerning the interchangeability of adiabatic and particle number limit will not be addressed here.

To give an informal explanation of our main theorem we introduce the stationary pendant to \eqref{eq_GP} which reads
\begin{align} \label{eq_GPtindep}
-\Delta \psi_{E,s}+V_{s}\psi_{E,s}+E\psi_{E,s}+b|\psi_{E,s}|^{2}\psi_{E,s}=0\, .
\end{align}
Its solutions are referred to as \emph{ground states} since they solve the Euler-Lagrange equation for the Gross-Pitaevskii energy functional
\begin{align}
I[\psi_{E,s}]:= \int_{\R^{3}} d^{3}x \ \left(\frac{1}{2}|\nabla \psi_{E,s}|^{2} +  V_{s}|\psi_{E,s}|^{2}+\frac{b}{4} |\psi_{E,s}|^{4} \right),
\end{align}
where $\|\psi_{E,s}\|_{2}^{2}=\eta$ is fixed. We assume that the initial data $\Psi_{0}$ for (\ref{eq_GP}) is small in a suitable sense (equivalently, $\Psi_{0}$ could be rescaled at the price of choosing the parameter $b$ to be small instead). In addition, we assume that the linear Hamiltonian \mbox{$-\Delta+ V_{s}$} admits exactly one bound state for each $s$, that is, the trapping potential $V_{s}$ is supposed to be weak. After adding the nonlinearity this bound state bifurcates into a whole manifold of ground states as will be shown below.

Our result can now be described as follows: Under the assumption that $\Psi_{0}$ belongs to the ground state manifold we prove that, up to phase and uniformly in $s$,  $\Psi_{s}$ converges to an element in the ground state manifold with equal mass, i.e. $L^2-$norm of the solution, as $\varepsilon \searrow 0$. In fact, the error term will be $\mathcal{O}(\eps)$ and is thus reminiscent of linear adiabatic theorems in presence of a gap condition.

From a physical perspective the existence of an adiabatic theorem for the Gross-Pitaevskii equation is to be expected and has been observed in interference experiments \cite{Anderson, Orzel}. Mathematically however, the non-autonomous setting considered here - contrary to the autonomous case (e.g. \cite{Soffer:1990vb, WeinsteinII, yautsai, froehlich, weinsteinIII}) - has not yet been subject to intensive investigations: A \emph{space}-adiabatic theorem was found in \cite{Salem08}. The result closest to ours is \cite{Sparber}, where the equation of the form 
\begin{align*}
 \i \varepsilon\partial_s \Psi_s = -\Delta\Psi_s+V_s\Psi_s+
b\varepsilon |\Psi_s|^{2\sigma}\Psi_s \, ,
\end{align*} was considered. Here the nonlinearity is of the size $\varepsilon$, which goes to zero as $\varepsilon$ tends to zero.

Interestingly, the techniques we apply to prove our theorem differ from the ones commonly used in the linear case. The main difficulty is that by linearizing the Gross-Pitaevskii equation around a ground state one obtains a generator which does no longer give rise to a contracting evolution on $L^{2}(
\mathbb{R}^{3})$ and therefore makes it more involved to estimate error terms. This problem can be dealt with by a bootstrap argument which uses the dispersive behavior of the {linear} Schr\"odinger equation; see e.g. \cite{Schlag, Goldberg:2006wz}. For the related nonlinear problems, see the results in \cite{BusP1992, RoWe, TsaiYau02, MR1992875, BuSu, GS07, G1, Cuccagna:03, CuccKirr, Cucca08, SW-PRL:05, NakanTsai12, ZwHo07, Cuccagna:15}.

The organization of this article is as follows. We start in Section~\ref{exgsmfd} we establish the existence and regularity of a ground state manifold for Equation \eqref{eq_GPtindep}. This enables us to give a precise statement of the main theorem in Section~\ref{chap4:mainthm}. After studying, in Section \ref{subsec_linop}, various properties of the linearized operator, obtained by linearizing around the ground state, we reformulate the main Theorem into Theorem \ref{refor} in Section \ref{sec:refor}. Theorem \ref{refor} will be proved in Sections \ref{chap4:main} and \ref{sec:full}.
Various technical estimates will be in Appendices. 

Throughout the paper we use the standard notation for the weighted Sobolev spaces $$H^{2,\sigma}(\R^{3}):=\{ \phi: \R^{3}\to \C |  \|\phi\|_{H^{2,\sigma}} := \| \japx^{\sigma}\phi \H < \infty \},\ \ \japx := \sqrt{1+|x|^{2}}.$$

\section{Ground state manifold: existence and regularity} \label{exgsmfd}
\subsection{Hypotheses on the potential}

We start with the general assumptions for the potential $V_{s}$. 

{\emph{
\begin{description} 
\item [$(\mathrm{H}_{\mathrm{d}})$]  The potential $V_{s}(x),\ s\in [0,1]$ and $x\in \mathbb{R}^3$, satisfies $V_{\cdot}\in C([0,1];H^{2,\sigma}(\R^{3}))\cap C^{2}([0,1];L^{\infty}(\R^{3})) $ for a $\sigma > 2$. 
\item[$(\mathrm{H}_{\mathrm{e}})$]  {For every $s\in [0,1]$, $-\Delta+V_{s}$ admits exactly one eigenstate $v_{*,s}$, with eigenvalue $-E_{*,s}<0$ separated from the rest of the spectrum of $-\Delta+V_{s}$, by a margin uniformly in $s$: there is $G_{0}>0$ such that $E_{*,s}\geq G_{0}$ for all $s$.}
\item[$(\mathrm{H}_{\mathrm{r}})$]  For every $s \in [0,1]$, $V_{s}$ admits no zero energy resonance, that is, the equation 
\begin{align*}
\left(-\Delta + V_{s}  \right) g=0
\end{align*}
 admits no distributional solution $g\notin L^{2}(\R^{3})$ such that $\japx^{-\beta}g \in L^{2}(\R^{3})$ for every $\beta> 1/ 2$. 
\end{description}	
}}


\subsection{Ground state manifold}
We present the proposition to establish the existence of a curve of constant mass in the manifold of instantaneous stationary states for Equation \eqref{eq_GP}. More generally, our result yields a differentiable manifold of nonlinear ground states. Before stating it we introduce some notation. By
\begin{align}
P_{H_{s}}^d&:= \ket{v_{*,s}}\bra{v_{*,s}}\, , \label{Pdisc} \\ 
P_{H_{s}}^c&:=\id -P_{H_{s}}^d \label{Pcont}
\end{align}
we denote the spectral projections onto the eigenvector space of $-\Delta +V_{s}$ and its orthogonal complement. Moreover, we declare that a subindex in Landau's $\mathcal{O}$-symbol denotes the space in which the statement is to be understood.

\begin{prop}\label{groundstatefamily} {~} For any fixed $s\in [0,1]$, $\eta \ll1$, $0\leq({E_{*,s}}-E)/{b} \ll 1$  and $l \in \R$, we have
\begin{enumerate}
\item The ``time-independent'' Gross-Pitaevskii equation~(\ref{eq_GPtindep}) admits a family of nonlinear ground states $\psi_{E,s}>0$ satisfying the equation
\begin{align}
(-\Delta+V_s+E)\psi_{E,s}+b(\psi_{E,s})^3=0,
\end{align}
and they bifurcate from the zero solution: 
\begin{align*}
\psi_{E,s}= \sqrt{\frac{E_{*,s}-E}{b}}\frac{1}{\sqrt{\langle v_{*,s}^2,v_{*,s}^2\rangle}}v_{*,s}+\mathcal{O}_{H^{2,l}}(E_{*,s}-E)\, .
\end{align*}
In fact, $\psi_{E,s}$ is analytic in $\sqrt{\frac{E_{*,s}-E}{b}}$ and $P_{H,s}^c\psi_{E,s}=
\mathcal{O}_{H^{2,l}}\left(\left(\frac{E_*-E}{b}\right)^\frac{3}{2}\right)$.
\item The ground states $\psi_{E,s}$ form a two-dimensional Banach manifold $\mathcal{M}\subset H^{2,l}(\R^{3})$. For fixed $s$ the assertions in (i) hold and the map $s\mapsto \psi_{E,s} \in H^{2,l}(\R^{3})$ is $C^{2}$.
\item There exists a unique positive family of ground states $s\mapsto \psi_{E_s, s} \in C^2([0,1];H^{2,l}(\R^3))$ with constant mass, $\|\psi_{E_{s},s} \|_2^2\equiv\eta$.
 \end{enumerate}
\end{prop}
The results are well known, see e.g. \cite{yautsai}. Hence we choose to skip the details.

\section{Main theorem}\label{chap4:mainthm}
The notion of a family of ground states allows to formulate the following adiabatic theorem, which is the main result.

\begin{thm} \label{mainthm} 
Let $\Psi_{0}=\psi_{E_{0},0}$ with $\|\psi_{E_{0},0}\|_{2}^{2}=\eta\ll 1$ as above and $\eps \ll 1$. Then Equation (\ref{eq_GP}) possesses a unique solution $s\mapsto \Psi_{s}$ in $C^{1}([0,1];H^{2}(\R^{3}))$ with the property that
\begin{align*}
\sup_{0\leq s \leq 1}\| \Psi_{s} - \e^{-\i\zeta_{s} }\psi_{E_{s},s}\H \lesssim \eps\, .
\end{align*}
Here, $\zeta_{s}:=\xi^{\eps}_{s}+\frac{1}{\eps}\int_{0}^{s}E_{s'}ds'$ and
 $\xi_{s}^{\eps}$ is a real function, uniformly bounded in $s$ and $\eps$.
\end{thm}
The theorem will be reformulated into Theorem \ref{refor} below.

The unesthetic factor $\e^{\i \zeta_{s}}$ is avoided by going over to projectors. Dirac notation allows us to formulate the following immediate corollary of  Theorem \ref{mainthm}.
\begin{cor} Under the assumptions in Theorem \ref{mainthm},
\begin{align*}
\sup_{0\leq s \leq 1}\| \ket{\Psi_{s}} \bra{\Psi_{s}}-  \ket{\psi_{E_{s},s}}\bra{\psi_{E_{s},s}} \|_{L^2\rightarrow L^2} \lesssim \eps\, .
\end{align*}
\end{cor}

For a complete proof we need to prove the local wellposedness of the solution in space $H^{2}(\mathbb{R}^3).$ But here we choose to skip it by the fact it is straightforward in the present setting, see also e.g. \cite{Ginibrevelo1, Ginibre:155598}.

The following well-known results will be used often. Their proofs are standard, hence omitted.
\begin{lem} \label{lem_norm} 
\begin{enumerate}
\item[]
\item $\|\phi\|_{H^2} \simeq \|\phi\|_2 + \|(-\Delta+V)\phi \|_2\, .$
\item  For any $l \in \R$
\begin{align*} 
\|\phi\Hl&\simeq \|\langle x \rangle ^{l}\phi\|_2 + \|\langle  x\rangle ^{l}\Delta \phi \|_2 \, .
\end{align*}
\item
Product estimates, recall that by definition $\| \phi\|_{W^{2,1}}:= \|\phi\|_{1}+\|\Delta \phi\|_{1}$,
\begin{align} \label{prodest1}
\|{\phi}\chi\|_{{W}^{2,1}} &\lesssim \|\phi\H \|\chi \H\, , \\
\label{prodest3} \|{\phi}\chi\H&  \lesssim \|{\phi}\H \|\chi\H \, ,\\
\label{prodest5}
\| \phi \chi \|_{H^{2,l}} &\lesssim \| \phi \|_{H^{2,l}} \| \chi \|_{H^{2,l}}\, , \hspace{0.5in} (\mathrm{if } \ l \geq0).
\end{align}
\end{enumerate}
\end{lem}

\section{Linearization around the ground state}   \label{subsec_linop}
We start with linearizing around the ground state and make an ansatz\footnote{The ansatz is arguably naive as it does not reflect possible geometric phase changes (in the sense of Berry's phase). This will be made good for in Lemma~\ref{lem_decomp}.} 
\begin{align} \label{naive}
 \Psi_s = e^{\frac{\i}{\epsilon}\int_0^sE_{s'}ds'}(\psi_{E_{s},s} + \varphi_s) \, 
\end{align}
to derive an equation for $\varphi_s$. Note that since the nonlinearity in (\ref{eq_GP}) is not complex analytic in the
wave function $\Psi_s$, the linearized operator will only be real linear. It is
thus favorable to adopt the notation 
\begin{align} \label{GP:convention}
\vec{\varphi}_s = \begin{pmatrix} \mathfrak{R}\varphi_s \\
\mathfrak{I}\varphi_s\end{pmatrix}=\begin{pmatrix}\varphi_{1,s} \\ \varphi_{2,s}
\end{pmatrix}\end{align} 
and likewise for any other complex quantities. Plugging (\ref{eq_decomp}) into (\ref{eq_GP}) to find
\begin{align} 
 \dot{\vec{\varphi}}_s&=-\frac{1}{\epsilon}J \begin{pmatrix} -\Delta+
\V+E_{s}+3b\psi_{E_{s},s}^2 & 0 \\ 0 & -\Delta+
\V+E_{s}+b\psi_{E_{s},s}^2 \end{pmatrix}\vec{\varphi}_s \label{eq_definingL} \\
&\phantom{aaa}-\frac{d}{ds}\vec{\psi}_{E_{s},s} -\frac { 1 } {
\epsilon}J\begin{pmatrix}b\psi_{E_{s},s}|\vec{\varphi}_s|^2 +
2b\psi_{E_{s},s}({\varphi_{1,s}})^2+b|\vec{\varphi}_s|^2\varphi_{1,s} \label{eq_disp0.1} \\ 2b\psi_{E_{s},s} \varphi_{1,s}
\varphi_{2,s} + b|\vec{\varphi}_s|^2 \varphi_{2,s}\end{pmatrix} \\ &=:
\frac{1}{\epsilon}L_{E_s,s} \vec{\varphi}_s-\frac{d}{ds}\vec{\psi}_{E_{s},s}-\frac { 1 } { \epsilon}N(\vec{\psi}_{E_{s},s},\vec{\varphi}_s) \, . \label{eq_disp1}
\end{align}
Here the linear operator $L_{E,s}$ is naturally defined by (\ref{eq_definingL}) and the nonlinearity $N=N(\vec{\psi}_{E_{s},s},\vec{\varphi}_s)$ by (\ref{eq_disp0.1}). We used the notation
\begin{align*}
J:=\begin{pmatrix} 0 & -1 \\ 1 & 0 \end{pmatrix}\, ,
\end{align*}
and moreover $L_{E,s}$ is considered as an unbounded operator on the Hilbert space $L^{2}(\R^{3})\oplus L^{2}(\R^{3})$. To facilitate later discussions we define operators $L_{E,s}^{+}, L_{E,s}^{-}$ so that $L_{E,s}$ takes the form
\begin{align} \label{eq_lin}
L_{E,s} =& \begin{pmatrix}
0 & L^-_{E,s} \\
-L^+_{E,s} & 0
\end{pmatrix}.
\end{align}

Now we study the eigenvalues of the operator $L_{E,s}$. Compute directly to find
\begin{align*}
L^-_{E,s}\psi_{E,s} =0\, .
\end{align*}
Hence $
\begin{pmatrix}0 , \psi_{E,s}
\end{pmatrix}^\intercal
$
is an eigenvector of $L_{E,s}$ with eigenvalue $0$.
Differentiation of the left hand side with respect to $E$ yields
\begin{align*}
L^+_{E,s} \partial_{E_{}}\psi_{E,s}= -\psi_{E,s} \, .
\end{align*}
It follows that  $\begin{pmatrix}- \partial_{E}\psi_{E,s} , 0 \end{pmatrix}^{\intercal}$ is an associated generalized eigenvector of $\begin{pmatrix} 0 , \psi_{E,s} \end{pmatrix}^{\intercal}$ for $L_{E,s}$. 

The Riesz projection for the linear operator $L_{E,s}$ take the following form.
\begin{lem}\label{chapGP:projexpr} $0$ is the only eigenvalue of $L_{E,s}$ in the ball of radius $G_{0}/2$ around zero (cf. $(\mathrm{H}_{\mathrm{e}})$). More precisely, if $\Gamma$ parametrizes its boundary in counterclockwise direction then
\begin{align} \label{eq:proj}
P_{E,s}^d:=-\frac{1}{2\pi \i} \oint_{\Gamma}({L_{E,s}-z})^{-1}dz = \frac{2}{\partial_{E_{}} \|\psi_{E,s}\|_2^2}\left( 
\Ket{\begin{matrix}\partial_{E_{}}\psi_{E,s} \\ 0 \end{matrix}} 
\Bra{\begin{matrix} \psi_{E,s} \\ 0 \end{matrix}} 
+
\Ket{\begin{matrix}0 \\ \psi_{E,s} \end{matrix}} 
\Bra{\begin{matrix} 0 \\ \partial_{E_{}}\psi_{E,s} \end{matrix}}
\right) \, .
\end{align}
\end{lem}
The result is well known, can be found in, for example, \cite{GS07}.

\section{Reformulation of Theorem~\ref{mainthm}} \label{sec:refor}
We start with decomposing the solution $\Psi_s$ into different components, according to the spectrum of $L_s$. It relies on the following lemma to decompose the solution. 
In figurative language what the lemma says is the following: As long as $\varphi_s$ in Ansatz~(\ref{naive}) is sufficiently small, then, at the cost of introducing an additional phase $\gamma_s^\eps$, we can ``shadow'' $\psi_{E_s,s} \in \mathcal{M}$ by $\psi_{E_s^\eps,s}$ such that 
\begin{align} \label{eq_decomp}
 \Psi_s = e^{-\frac{\i}{\epsilon}(\int_0^sE^\eps_{s'}ds' - \gamma^\eps_s)}(\gs + \phi_s)
\end{align}
with $P_{E_s^\eps,s}^d \phi_s=0$.
\begin{lem}  \label{lem_decomp}

\noindent For any $\phi$ with $\|\phi\|_{2} \ll 1$ there exist parameters $\hat{E}=\hat{E}(E,s,\phi)$ and $\hat{\gamma}=\hat{\gamma}(E,s,\phi)\in \mathbb{R}$, with $\hat{E}(E,s,0)=E$ and $\hat{\gamma}(E,s,0)=0$, such that if $$\Psi_s=\phi_{E,s}+\phi$$ then
\begin{align*}
\Psi_s=
\e^{i\hat{\gamma}}\left(\psi_{\hat{E},s}+\phi_{\hat{E},s}\right)\, ,
\end{align*}
where $\phi_{\hat{E},s}$ lies in the continuous subspace of $L_{\hat{E},s}$. The dependence of $\hat{E}$, $\hat{\gamma}$ on $E$, $\phi$ is smooth; the dependence on $s$ is $C^{2}$.
\end{lem}
The proof of the lemma can essentially be found in \cite{yautsai}. 

Applying Lemma \ref{lem_decomp}, we reformulate Theorem~\ref{mainthm} in terms of estimates on various components of (\ref{eq_decomp}). Plug \eqref{eq_decomp} into \eqref{eq_GP}  to find,
\begin{align} \label{eq_disp}
 \dot{\vec{\phi}}_s&= 
\frac{1}{\epsilon}L_s \vec{\phi}_s-\frac{1}{\eps} \dot{\gamma}^\eps_s J\dist-\frac{d}{ds}\gsv-\frac{1}{\varepsilon}\dot{\gamma}^\eps_t
J\gsv -\frac { 1 } { \epsilon}N(\gsv,\vec{\phi}_s)\, .
\end{align}
Recall that $\vec{\phi}:=\left(
\begin{array}{lll}
\mathcal{R}\phi\\
\mathcal{I}\phi
\end{array}
\right)$ by the convention in \eqref{GP:convention}, $L_{s}:=L_{E_s^\eps,s}$ and the nonlinearity is defined in \eqref{eq_disp1}, and for later use, $P_{s}^{d}:=P^{d}_{E_{s}^{\eps},s}$. It is not hard to see that our initial condition $\vec\phi_0$ satisfies
\begin{align}
\vec\phi_0=0.
\end{align}

As a first consequence we derive equations for $\dot{E}_{s}^{\eps}, \dot{\gamma}_{s}^{\eps}$, the \emph{modulation equations}. To that end recall the condition $P_{s}^{d}\vec{\phi}_{s}=0$, which, by (\ref{eq:proj}) amounts to 
\begin{align*}
\Braket{ \vec{\phi}_{s} | \begin{matrix} \psi_{E,s} \\ 0 \end{matrix} } = \Braket{ \vec{\phi}_{s} | \begin{matrix} 0\\ \partial_{E}\psi_{E,s}  \end{matrix} } =0 \, .  
\end{align*}
Apply these two orthogonality conditions to Equation \eqref{eq_disp} to derive
\begin{align}
&\dot{E_{s}^\eps}\big[ \langle \partial_{E_{}}\gs, \phi_{1,s}\rangle-\langle \gs,\partial_{E_{}}\gs \rangle\big]+ \frac{\dot{\gamma}_s^\eps}{\eps}\langle \gs , \phi_{2,s} \rangle \nonumber\\
 = & -\langle \partial_s \Psi_{E_t^{\epsilon},s}|_{t=s}, \phi_{1,s}\rangle+   \langle \gs,\partial_s \Psi_{E_t^{\epsilon},s}|_{t=s}\ \rangle +\frac 1 \eps \langle \gs,N(\gsv,\dist)\rangle,
\label{modeq1} \\ 
&\dot{E_{s}^\eps} \langle \partial^2_{E_{}}\gs, \phi_{2,s} \rangle-\frac{\dot{\gamma}^\eps_s}{\eps}\big[ \langle \partial_E \gs, \phi_{1,s}\rangle +\langle \partial_{E_{}}\gs,\gs\rangle\big]  \nonumber\\
=& \frac 1 \eps \langle \partial_{E_{}}\gs,N(\gsv,\dist)\rangle -\langle \partial_s \partial_{E_{}}\gs, \phi_{2,s}\rangle  \, .\label{modeq2}
\end{align}

We continue to transform the equation for $\dot{\vec{\phi}}_s$ in \eqref{eq_disp}. To remove the term $-P_s^c\frac{d}{ds}\gsv$, which is considered the main term, from the right hand, we make a refined decomposition for the function $\dist$ as follows
\begin{align} \label{finerdecomp}
\dist = \epsilon L_s^{-1}P_s^c\frac{d}{ds}\gsv+{\vec{\tilde{\phi}}}_s = \epsilon L_s^{-1}P_s^c\partial_s\gsv+{\vec{\tilde{\phi}}}_s\, .
\end{align} The function $L_s^{-1}P_s^c\partial_s\gsv$ is well defined since,
by Lemma \ref{chapGP:projexpr},  the operator $L^{-1}_{s} : P_{s}^{c}(L^{2}(\R^{3})) \to P_{s}^{c}(L^{2}(\R^{3}))$ is well-defined and bounded. The second identity in \eqref{finerdecomp} follows from the explicit form of the projection \eqref{eq:proj}. 

Plug the Decomposition \eqref{finerdecomp} into \eqref{eq_disp} to yield the equation for ${\vec{\tilde\phi}}_s$,
\begin{align*}
 \dot{{\vec{\tilde\phi}}}_s&=\frac{1}{\epsilon}L_s{\vec{\tilde\phi}}
_s-\frac{\dot{\gamma}^{\eps}_s}{\eps}J\distt-\dot{\gamma}^{\eps}_sJL_s^{-1}P_s^c\partial_{s}\gsv-P_s^d\frac { d } { ds } \gsv -\frac{\dot{\gamma}^{\eps}_s}{\epsilon}J\gsv - \epsilon
\frac{d}{ds}(L_s^{-1}P_s^c\partial_{s} \gsv) \\
&\phantom{aaa}-
\frac{1}{\epsilon}N(\gsv,\varepsilon
L_s^{-1}P_s^c\partial_{s}\gsv+{\vec{\tilde\phi}}_s) \, .
\end{align*}
Apply $P_{s}^{c}$ on both sides to see certain term vanish, and use $\dot{{\vec{\tilde{\phi}}}}_s =\dot{P}_s^c
{\vec{\tilde{\phi}}}_s+{P}_s^c
\dot{{\vec{\tilde\phi}}}_s=-\dot{P}_s^d
{\vec{\tilde\phi}}_s+{P}_s^c
\dot{{\vec{\tilde\phi}}}_s$ to obtain
\begin{align}
 \dot{{\vec{\tilde\phi}}}_s
&= \frac{1}{\epsilon}L_s{\vec{\tilde\phi}}
_s-\frac{\dot{\gamma}^{\eps}_s}{\eps}P_s^cJ\distt-\dot{P}^d_s{\vec{\tilde\phi}}_s-\dot{\gamma}^{\eps}_sP_s^cJL_s^{-1}P_s^c\partial_{s}\gsv-\frac{\dot{\gamma}_s^{\eps}}{\epsilon}\underbrace{P_s^cJ\gsv}_{=0} -
\epsilon
P_s^c\frac{d}{ds}(L_s^{-1}P_s^c\partial_{s} \gsv)\nonumber \\ &\phantom{aaa}-
\frac{1}{\epsilon}P_s^cN(\gsv,\varepsilon
L_s^{-1}P_s^c\partial_{s}\gsv+{\vec{\tilde\phi}}_s) \, .\label{dispeq11}
\end{align}
The initial data for the equation above is, by the fact $\vec\phi_0=0$ and \eqref{finerdecomp},
\begin{align} \label{gpinitialdata}
{\vec{\tilde{\phi}}}_{0}=-\eps L_{0}^{-1}P_{0}^{c}\partial_{s}\gsv|_{s=0}=-\eps L_{0}^{-1}P_{0}^{c}\partial_{s}\vec{\psi}_{E_{s},s}|_{s=0}\, .
\end{align}

We now reformulate Theorem \ref{mainthm}:
\begin{thm}[Reformulation of Theorem~\ref{mainthm}]\label{refor} The function $\distt$, the scalar functions $E_{s}^{\eps}, \gamma_{s}^{\eps}$ satisfy the following estimates:
\begin{align}
\sup_{0\leq s \leq 1}|E_{s}^{\eps}-E_{s}| &\lesssim \eps \, \label{distEest} , \\
\sup_{0\leq s \leq 1}|\dot{\gamma}_{s}^{\eps}|& \lesssim \eps^{2} \label{gammaquadest}\, ,\\
\sup_{0\leq s \leq 1} \| {\vec{\tilde{\phi}}}\|_{H^2} &\lesssim \eps \label{distestn1} \, ,\\
\sup_{0\leq s \leq 1} \|  {\vec{\tilde{\phi}}}\|_{H^{2,-\sigma}} &\lesssim \eps^2\, .
\end{align}

\end{thm}
The theorem will be proved for a short time interval $[0,s_0]$, with $s_0$ small but independent of $\epsilon,$ in Section \ref{chap4:main} below, and proved in the interval $[0,1]$ in Section \ref{sec:full}.

Clearly, Theorem~\ref{refor} implies Theorem~\ref{mainthm}: with $\xi^{\eps}_{s}:=\frac 1 \eps \int_{0}^{s}\left(E_{s'}^{\eps}-E_{s'}\right)ds'$ it follows that 
\begin{align*}
\sup_{0\leq s \leq 1}\| \Psi_{s}- \e^{-\i(\xi_{s}^{\eps}+\frac{1}{\eps}\int_{0}^{s}E_{s'}ds')}\psi_{E_{s},s}\H&=\sup_{0\leq s \leq 1}\|\e^{\i \frac{\gamma_{s}^{\eps}}{\eps}}(\psi_{E_{s}^{\eps},s}+\phi_{s})- \psi_{E_{s},s}\H \\ 
& \lesssim \sup_{0\leq s \leq 1} \| \psi_{E_{s}^{\eps},s} - \psi_{E_{s},s} \H + \sup_{0\leq s \leq 1}\|\phi_{s} \H +\eps \\
&\lesssim \eps \, .
\end{align*}
The first inequality makes use of (\ref{gammaquadest}), the second of Proposition~\ref{groundstatefamily} in combination with (\ref{distEest}) as well as (\ref{distestn1}).

\section{Proof of Theorem~\ref{refor} for small time $[0,s_0]$} \label{chap4:main} 

Mathematically, the main work for the proof of Theorem~\ref{refor} lies in the demonstration of its validity on a small interval $[0,s_{0}]$, with $s_{0}$ being independent of $\eps$. Here we need $s_0$ small enough so that some Fixed-Point-Theorem-type argument can be applied, see the choice of (or two conditions on) $s_0$ in \eqref{deltachoice} and \eqref{chooseDeltaA} below.

We begin with presenting the main ideas. The core of the proof is a bootstrap argument. Specifically, define a locally controlling function $M_{s}^{(l)}$ as
\begin{align*}
\Ml:&=\sup_{0\leq s' \leq s}\|\dists\Hls
\end{align*}
and a globally controlling function $M_{s}^{(\mathrm{g})}$ as
\begin{align*}
\M:&=\sup_{0\leq s' \leq s}\|\dists\H\, ,
\end{align*}
cf. Equation \eqref{finerdecomp}. Recall that $\sigma>2$.

To start the bootstrap arguments, we use that,
by \eqref{gpinitialdata}, the initial data satisfies the conditions
\begin{align*}
\|{\vec{\tilde{\phi}}}_{0}\Hls &\leq \eps \|L_{0}^{-1}P_{0}^{c}\partial_{s}\vec{\psi}_{E_{s},s}|_{s=0}\|_{H^{2}\cap W^{2,1}}\, , \\
\|{\vec{\tilde{\phi}}}_{0}\H &\leq \eps \|L_{0}^{-1}P_{0}^{c}\partial_{s}\vec{\psi}_{E_{s},s}|_{s=0}\H\, ,
\end{align*}
resp. 
\begin{align*}
{\vec{\tilde{\phi}}}_{0}=0\hspace{.5in} \text{ if }L_{0}^{-1}P_{0}^{c}\partial_{s}\vec{\psi_{E_{s},s}}|_{s=0}=0\, .
\end{align*}
Hence there is a maximal $0<\tau $ such that the $(\mathrm{B}_{\mathrm{l}}$, $\mathrm{B}_{\mathrm{g}})$ resp. $(\mathrm{B}_{\mathrm{l}}'$, $\mathrm{B}_{\mathrm{g}})$ conditions, to be defined below, are satisfied with $s=\tau$ as long as $\eps \ll 1$. Here the $(\mathrm{B}_{\mathrm{l}}$, $\mathrm{B}_{\mathrm{g}})$ resp. $(\mathrm{B}_{\mathrm{l}}'$, $\mathrm{B}_{\mathrm{g}})$ conditions are defined as
\begin{itemize}
\item[$(\mathrm{B}_{\mathrm{l}})$] \hspace{1in}
$M_{{s}}^{(\mathrm{l})} \leq 2A\eps \| L_0^{-1}P_0^c \partial_s \vec{\psi}_{E_{s},s} |_{s=0} \|_{{H^2} \cap W^{2,1}}+\epsilon $  ,
\item[$(\mathrm{B}_{\mathrm{g}})$] \hspace{1in}
$M_{{s}}^{(\mathrm{g})}\leq  \eps^{\frac 2 3}$\, ,
\end{itemize}
where $A>1$ is the constant in the dispersive Estimate \eqref{dispest} below. Note that if $L_0^{-1}P_0^c \partial_s \vec{\psi}_{E_0,s}|_{s=0}=0$ then \Bsl $\ $ is replaced by
\begin{itemize}
\item[$(\mathrm{B}_{\mathrm{l}}')$] \hspace{1in}
$M_{{s}}^{(\mathrm{l})} \leq \eps$\, .
\end{itemize}

With these at hand we may turn to the analysis of $\distt$ itself. The key fact is that $\distt$ lies in the continuous subspace of the linear operator $L_{s}$, which allows us to apply dispersion estimates. which in turn is generated by some linear operator approximately $L_s$. This in turn is generated by some linear operator which approximates $L_s.$
Together with bootstrap assumption \Bs, this will enable us to improve the estimates for $\distt$ on the small interval $[0,\tau]$, as long as $\tau\leq s_0$ with $s_0>0$ being small and independent of $\epsilon$. The result is the following proposition.
\begin{prop}\label{gpshortestlem}There exists $s_{0}>0$, independent of $\epsilon$ (provided that it is sufficiently small), such that if $(\mathrm{B}_{\mathrm{l}}$, $\mathrm{B}_{\mathrm{g}})$ resp. $(\mathrm{B}_{\mathrm{l}}'$, $\mathrm{B}_{\mathrm{g}})$ hold for $s\leq s_{0}$, then the better estimates
\begin{align}
M_{s}^{(\mathrm{l})} &\leq  A\eps  \| L_0^{-1}P_0^c \partial_s \vec{\psi}_{E_0,s} |_{s=0} \|_{{H^2} \cap W^{2,1}}+\frac{2}{3}\epsilon\, , \label{bootstraploc}\\
\M& \lesssim \eps \label{bootstrapglob}
\end{align}
are true for all $\eps \ll 1$. If $L_0^{-1}P_0^c \partial_s \vec{\psi}_{E_0,s}|_{s=0}=0$ then \eqref{bootstraploc} is replaced by
\begin{align} \label{locestrepl}
M_{s}^{(l)}\lesssim \eps^{2}\, .
\end{align}
\end{prop}
The proposition will be proved in the rest of the section.

Assuming Proposition \ref{gpshortestlem} holds, then we can prove \eqref{distestn1} of Theorem \ref{refor} by a continuity argument.

In order to prove \eqref{bootstraploc}, \eqref{bootstrapglob} we start with establishing controls for the modulation parameters. This will also yield \eqref{distEest} and \eqref{gammaquadest} of Theorem \ref{refor}.

\subsection{Control of modulation parameters $E_s^{\epsilon}$ and $\gamma_{s}^{\epsilon}$, proof of \eqref{distEest} and \eqref{gammaquadest}} \label{contrmodpar}
Recall the decomposition of the solution in \eqref{eq_decomp}. The parameters $E_{s}^{\epsilon}$ and $\gamma_{s}^{\epsilon}$ satisfy the equations \eqref{modeq1} and \eqref{modeq2}. The function $\dist$ is further decomposed in \eqref{finerdecomp}, and the function ${\vec{\tilde{\phi}}}_s$ satisfies Equation \eqref{dispeq11}.

We start with some preliminary estimates. It is easy to see that every scalar product in \eqref{modeq1} and \eqref{modeq2} which involves $\phi_s$ is of order $\eps$ whenever assumption $(\mathrm{B}_{\mathrm{l}})$ resp. $(\mathrm{B}_{\mathrm{l}}')$ holds. For $\eps \ll 1$, it follows that
\begin{align}
|\dot{E}^\eps_s|&\lesssim 1\, ,  \\
|\dot{\gamma}^\eps_s|&\lesssim \epsilon^2 \, . \label{gammaest} 
\end{align}
Note that \eqref{gammaest} is the desired bound in \eqref{gammaquadest}.

Now we prove \eqref{distEest}. Recall that the scalar function $E_s$, which is independent of $\epsilon$, is the function satisfying 
$\|\Psi_{E_s,s}\|^2_2=\|\Psi_{E_0,0}\|^2_2=\eta, $ see (iii) of Proposition \ref{groundstatefamily}.
In the next result we measure the difference between $E_s$ and $E_s^{\epsilon}$, or is to prove \eqref{distEest}.
\begin{lem}\label{energylincontrol} If \Bsl  \ resp. $(\mathrm{B}_{\mathrm{l}}')$ holds in $[0,s]$, then we have
\begin{align*} 
|E^{\eps}_{s}-E_{s}|\lesssim \eps\, ,
\end{align*}
uniformly in $s$.
\end{lem}
\begin{proof}
We apply a Gr\"onwall-type argument. The function $$f(E,s):=-\frac{\langle \psi_{E,s},\partial_{s}\psi_{E,s}\rangle}{\langle \psi_{E,s},\partial_{E}\psi_{E,s}\rangle}$$ is $C^{1}$ in $s$ and smooth in $E$ by Proposition~\ref{groundstatefamily}.
By \Bsl \ and (\ref{modeq1}) 
\begin{align}
\dot{E}^{\eps}_{s} = f(E_{s}^{\eps},s) + {\mathcal{O}}(\eps)\, ,\label{eq:EEpsS}
\end{align}
where $\mathcal{O}(\eps)$ is uniformly bounded in $s$. 

To derive an equation for $E_s$, we take a $s-$derivative on the identity $\langle \psi_{E_{s},s}, \psi_{E_{s},s}\rangle= \langle \psi_{E_{0},0}, \psi_{E_{0},0}\rangle=\eta$ to find
\begin{align}
\dot{E}_{s}= f(E_{s},s)\, .
\end{align}
Subtract this from \eqref{eq:EEpsS} to obtain
\begin{align}
\dot{E}^{\eps}_{s}-\dot{E}_{s}&= f(E_{s},s)-f(E^{\eps}_{s},s)+{\mathcal{O}}(\eps) \nonumber \\
&\leq C\left( |E^{\eps}_{s}-E_{s}| + \eps\right). \label{energyestimate}
\end{align}
Here $C$ is a constant independent of $\epsilon$. Using $E^{\eps}_{0}=E_{0}$ (see Theorem \ref{mainthm})
\begin{align*}
|E_{s}^{\eps}-E_{s} |\leq C \eps s
+C\int_{0}^{s}|E_{s'}^{\eps}-E_{s'}|ds'\, .
\end{align*}
This last inequality yields for $x(s):= \e^{-Cs}\int_{0}^{s}ds' \ |E_{s'}^{\eps}-E_{s'}| $
\begin{align*}
\dot{x}(s) \leq  C \eps s \e^{-Cs}
\end{align*}
which, after integration, implies the desired claim.
\end{proof}

\subsection{Proof of Proposition \ref{gpshortestlem}}
We start from the equation for $\dot{{\vec{\tilde\phi}}}_s$ in \eqref{dispeq11}.

In the proof we rely on propagator estimates generated by certain linear operators. Here we do not choose $L_s$ on the right hand side of \eqref{dispeq11} since its time-dependence will complicate our analysis. Instead we approximate it by a linear operator $-JH_{E_{s}^{\eps},0}$ with $H_{E_{s}^{\eps},0}$ defined as
\begin{align} \label{chapgp:refham}
H_{E^{\eps}_{s},0}:=-\Delta+
V_{0}-E^{\eps}_{s}=H_{0}-E^{\eps}_{s},
\end{align} and the linear Schr\"odinger operator $H_{0}$ naturally defined.
We rewrite the equation for $\dot{{\vec{\tilde\phi}}}_s$ as
\begin{align} \label{phitildeq11}
&\dot{{\vec{\tilde\phi}}}_s=-\frac{1}{\varepsilon}J(H_{E_{s}^{\eps},0}+\dot{\gamma}^{\eps}_s){\vec{\tilde\phi}}
_s   \\
& \phantom{aaa} +   \frac{1}{\epsilon}(L_s+JH_{E_{s}^{\eps},0}){\vec{\tilde\phi}}
_s    + \frac{\dot{\gamma}^{\eps}_s}{\eps}P_s^dJ\distt   -P_{s}^{d}\dot{P}^d_s{\vec{\tilde\phi}}_s \nonumber  \\
& \phantom{aaa}
-\dot{\gamma}^{\eps}_sP_s^cJL_s^{-1}P_s^c\partial_{s}\gsv 
-
\epsilon
P_s^c\frac{d}{ds}(L_s^{-1}{P_s^c\partial_{s}\gsv}) -\frac{1}{\epsilon}P_s^cN(\gsv,\varepsilon
L_s^{-1}P_s^c\partial_{s}\gsv+{\vec{\tilde\phi}}_s) \, . \nonumber 
\end{align}
The difference between $L_s$ and $-JH_{E_{s}^{\eps},0}$,
\begin{align}\label{chapgp:diffexpre}
L_s-\left(-JH_{E_{s}^{\eps},0}\right)=- J\begin{pmatrix} 
\V-V_{0}+3b\gs^2  & 0 \\ 0 & 
\V-V_{0}+b\gs^2 \end{pmatrix}\, ,
\end{align} 
is small if $s$ is small, and decays at spatial infinity. 

Apply the Duhamel's principle and apply $P_{H_{0}}^{c}$ on \eqref{phitildeq11} to find,
\begin{align} \label{initial}
P^c_{H_0}{{\vec{\tilde\phi}}}_s=& U_0(s,0)P^c_{H_0}{{\vec{\tilde\phi}}}_{0} \\ \label{lin}
&+ \int_0^sds' \ U_0(s,s') \Big(  \frac{1}{\epsilon}P^c_{H_0}(L_{s'}+H_{E^{\eps}_{s},0}){\vec{\tilde{\phi}}}
_{s'}+\frac{\dot{\gamma}^{\eps}_{s'}}{\eps}P_{H_0}^cP_{s'}^dJ\dists-P^c_{H_0}P^{d}_{s'}\dot{P}^d_{s'}{\vec{\tilde{\phi}}}_{s'}  \\
&\phantom{aaaaaaaaaaaaaaa} -\dot{\gamma}^{\eps}_{s'}P_{H_0}^cP_{s'}^cJL_{s'}^{-1}P_{s'}^c\partial_{s'}\gsvs-\epsilon P^c_{H_0}P_{s'}^c\frac{d}{ds'}(L_{s'}^{-1}{P_{s'}^c\partial_{s'} \gsvs}) \label{inhomo}
\\&\phantom{aaaaaaaaaaaaaaa}-\frac{1}{\epsilon}P^c_{H_0}P_{s'}^cN(\gsvs,\varepsilon
L_{s'}^{-1}P_{s'}^c\partial_{s'}\gsvs+{\vec{\tilde{\phi}}}_{s'}) \Big)\, .\label{nonlinear}
\end{align}

Here, the propagator $U_{0}(s,s'),\ s\geq s'\geq 0$ is generated by the linear operator $H_{E_{s}^{\eps},0}+\dot{\gamma}^{\eps}_{s}$ defined in \eqref{chapgp:refham}. Its mathematical definition is
\begin{align*}
\eps\partial_{s}U_{0}(s,s')=&-J(H_{E_{s}^{\eps},0}+\dot{\gamma}_{s}^{\eps})U_{0}(s,s')\, ,\\
U_{0}(s',s')=&\mathrm{Id}.
\end{align*}
The time decay estimates of various terms \eqref{initial}-\eqref{nonlinear} depend critically on the estimation of $U_{0}(s,s')$. Using the definition of $H_{E_{s}^{\eps},0}$, we cast the expression into a convenient form
\begin{align}\label{eq:relateHu}
U_{0}(s,s')=e^{\frac{1}{\epsilon} [\int_{s'}^{s} E_{z}^{\epsilon} \ dz-\gamma_{s}^{\epsilon}+\gamma_{s'}^{\epsilon}]J} e^{-\frac{1}{\epsilon} (s-s')(-\Delta+V_{0})J}.
\end{align} The first factor $e^{\frac{1}{\epsilon} [\int_{s'}^{s} E_{z}^{\epsilon} \ dz-\gamma_{s}^{\epsilon}+\gamma_{s'}^{\epsilon}]J}$ is a $2\times 2$ scalar unitary matrix, since $J$ is anti-self-adjoint.
Hence to estimate $U_{0}(s,s'),$ it suffices to estimate $e^{i\frac{1}{\epsilon} (s-s')(-\Delta+V_{0})}$. Moreover note that $U_{0}(s,s')$ commutes with $P_{H_{0}}^{d}$ and $P_{H_{0}}^{c}$, the projections onto the discrete and continuous subspace of $H_{0}$, respectively.

To estimate \eqref{initial}-\eqref{nonlinear} we rely on appropriate propagator estimates. Here to facilitate later discussions we consider cases more general than $e^{i\frac{1}{\epsilon} (s-s')(-\Delta+V_{0})}$, namely $e^{i\frac{1}{\epsilon} (s-s')(-\Delta+V_{\tau})},\ \tau\in [0,1].$

\begin{thm}[Goldberg]  \label{cor1} 
Under conditions $(\mathrm{H}_{\mathrm{r}}, \mathrm{H}_{\mathrm{d}})$ it holds for arbitrary $\tau \in [0,1]$ and $t \in \R$ that
\begin{align} \label{goldbergest}
\|\e^{-\i t H_\tau} P_{H_\tau}^c \|_{{L^1} \rightarrow L^\infty} &\lesssim |t|^{-\frac{3}{2}}\, \\
 \label{propest}
\|\e^{-\i\frac{t}{\eps}H_{\tau}}P^c_{H_\tau}\chi\H &\simeq \|P^c_{H_\tau}\chi\|_{H^2} \lesssim  \|\chi\|_{H^2}\, ,  \\
\|\e^{-\i\frac{t}{\eps}H_{\tau}}P^c_{H_\tau}\chi\Hls & \leq A \langle \frac{t}{\eps} \rangle^{-\frac 3 2}\|\chi\|_{{H^2}\cap W^{2,1}}\, , \label{dispest}
\end{align}
where $\|\chi\|_{H^2 \cap W^{2,1}}:=\| \chi \|_{H^2} + \|\chi\|_{W^{2,1}}$. The constant $A$ and the multiplicative constant in (\ref{propest}) can be chosen independent of $\tau$ and $t$.
\end{thm}
The proof is based on results in \cite{Goldberg:2006wz}, and will be given in Section \ref{sec:B}. 
\color{black}

We are now ready to estimate the various terms \eqref{initial}-\eqref{nonlinear}, the local estimates for (\ref{initial}-\ref{nonlinear}) are collected in the following lemma. Its proof is provided in the next subsection.
Recall that $\|\psi_{E_0,0}\|_2^2=\eta.$
\begin{lem}\label{lemtriplestar} Assume $(\mathrm{B}_{\mathrm{l}}$, $\mathrm{B}_{\mathrm{g}})$ resp. $(\mathrm{B}_{\mathrm{l}}'$, $\mathrm{B}_{\mathrm{g}})$ in the time interval $[0,s]$. We have that if $\eta,\ \epsilon \ll 1 $ then in $[0,\ s],$
\begin{align*}
\| (\ref{initial})\Hls & \leq A  \jte^{-\frac{3}{2}} \eps  \|   L_{0}^{-1}P_{0}^{c}\partial_{t}\gsv|_{t=0}\|_{H^{2}\cap W^{2,1}}   \simeq \jte^{-\frac{3}{2}}    \eps  \, , \\
\| (\ref{lin})\Hls & \lesssim {\delta}(0,s)\ (\eps \jte^{-\frac{3}{2}}+    \eps^{2}  ) M^{(\frac{3}{2},\mathrm{l})}_{s} + \eps^{2}\, , \\
\| (\ref{inhomo})\Hls & \lesssim \eps^{2} \, ,  \\
\| (\ref{nonlinear})\Hls & \lesssim \eps^{2} \, .
\end{align*}
 If $L_0^{-1}P_0^c \partial_s \vec{\psi}_{E_0,s}|_{s=0}=0$ then the first two estimates are replaced by
\begin{align*}
\| (\ref{initial})\Hls & = 0\, , \\
\| (\ref{lin})\Hls & \lesssim \delta(0,s)\ M_{s}^{(\mathrm{l})} +\eps^{2}   \, .
\end{align*}

\end{lem}
Here the constants $M^{(\frac{3}{2},\mathrm{l})}_{s}$ and $\delta(u,\tau)$ are defined as
\begin{align} 
M^{(\frac{3}{2},\mathrm{l})}_{s}:= &\sup_{0\leq s' \leq s} (\eps \langle
\frac{s'}{ \eps }\rangle^{-\frac{3}{2}}+\eps^{2} )^{-1} \| \dists \Hls,\label{eq:M32}\\
\delta(u,\tau):=&\sup_{u\leq s\leq \tau}[\|\V-V_{u}+3b\gs^2 \|_{H^{2,\sigma}}].\label{est:locality}
\end{align}

One last minor difficulty remains before proving Proposition \ref{gpshortestlem}. Recall that we need to prove that $\|\distt\|_{H^{2,-\sigma}}\leq \cdots$ using \eqref{initial}, while what appears on the left hand side of \eqref{initial} is $P^c_{H_0}\distt$, instead of the desired $\distt$. In the next lemma we show their $H^2$ and $H^{2,-\sigma}$-norms are equivalent. 
\begin{lem} \label{normequiv} Suppose $s_{0}$, $\eta$ and $\epsilon$ are sufficiently small. The for $0\leq s\leq s_{0},$ we have
\begin{align*}
\|P^c_{H_0}\distt \H&\simeq \| \distt \|_{H^2}\, , \\
\|P^c_{H_0}\distt \Hls &\simeq \|\distt \Hls \, .
\end{align*}
\end{lem}
The lemma will be proved in Section \ref{sec:C}.

Given Lemmata \ref{lemtriplestar} and \ref{normequiv}, we are ready to prove Proposition~\ref{gpshortestlem}.

\begin{proof}[Proof of Proposition~\ref{gpshortestlem}] We discuss first the case where $L_0^{-1}P_0^c \partial_s \vec{\psi}_{E_0,s}|_{s=0}\neq0$. 

Results in Lemmata \ref{lemtriplestar} and \ref{normequiv} imply that, for all sufficiently small $\eps$
\begin{align} \label{locest}
\| {\vec{\tilde{\phi}}}_{s} \Hls \leq C_{s_{0},A} \left(\eps \jte^{-\frac{3}{2}} +  \eps^2+ \delta(0,s)\ (\eps \jte^{-\frac{3}{2}}+\eps^{2})\Mlt \right)\, .
\end{align}
Here we have made the multiplicative constant $C_{s_{0},A}$ explicit in order to define our prescription for $\delta(0,s)$: We choose $s_0$ small enough so that Lemma \ref{normequiv} holds for $s\leq s_0$ and 
\begin{align}\label{deltachoice}
C_{s_{0},A}\delta(0,s_0) \leq 1/2\, .
\end{align}
\comment{

. It follows that
\begin{align*}
\left( \eps \jte^{-\frac{3}{2}} +\eps^{2}    \right)^{-1}\|\dist\Hls \leq C\left(1+\delta \Mltn \right)
\end{align*}
which, after taking the supremum and choosing $\delta\leq\frac{1}{2C}$ (and $s_{0}$ accordingly) yields }
Consequently, by the definition of $M^{(\frac{3}{2},\mathrm{l})}_{s}$ in \eqref{eq:M32}
\begin{align} \label{chapgp:boundM}
\Mltn \leq 2C_{s_{0},A}\, 
\end{align}
and therefore
\begin{align} \label{ineq_locest}
\| {\vec{\tilde{\phi}}}_s \Hls \leq 2C_{s_{0},A}\left( \eps \jte^{-\frac{3}{2}} + \eps^2\right)\, .
\end{align}
This, together with applying the results in Lemma \ref{lemtriplestar} to yield
\begin{align*}
\|{\vec{\tilde{\phi}}}_s\Hls \leq A\eps \| L_0^{-1}P_0^c \partial_s \vec{\psi}_{E_0,s} |_{s=0} \|_{{H^2} \cap W^{2,1}}+\tilde{C}_{s_{0},A}\left(\delta(0,s)\ \eps + \eps^2 \right)\, .
\end{align*}
In addition to the condition on $\delta(0,s)$ in \eqref{deltachoice} we require that $\delta$ also satisfies 
\begin{align}
\delta(0,s_0) \leq \frac{1}{2\tilde{C}_{s_{0},A}} .\label{chooseDeltaA}
\end{align} Hence the bootstrap assumption~\Bsl \ , for $s\in[0,s_{0}]$, is improved to the desired estimate\eqref{bootstraploc},
\begin{align} 
M_{s}^{(\mathrm{l})} \leq  A\eps  \| L_0^{-1}P_0^c \partial_s \vec{\psi}_{E_0,s} |_{s=0} \|_{{H^2} \cap W^{2,1}}+\frac{2}{3}\epsilon\, \nonumber
\end{align}
for all sufficiently small $\eps$.

Now we turn to estimating $\|{{\vec{\tilde{\phi}}}}_s\H$. By $(\mathrm{B}_{\mathrm{l}},$ $\mathrm{B}_{\mathrm{g}})$ and Lemma~\ref{energylincontrol}, for $s'\leq s_{0}$ it holds
\begin{align*}
\| N(\gsvs,{\vec{{\phi}}}_{s'})\H &\lesssim \eps^{2} \, .
\end{align*}
This together with \eqref{initial}-\eqref{nonlinear} and \eqref{ineq_locest} yields the desired estimate \eqref{bootstrapglob}
\begin{align}
\|{{\vec{\tilde{\phi}}}}_s\H \lesssim & \|{{\vec{\tilde{\phi}}}}_{0}\H + \int_0^sds' \ \Big(   \frac{1}{\epsilon}\|{\vec{\tilde{\phi}}}
_{s'}\Hls+\eps^{2}\|L_{s'}^{-1}P_{s'}^c\partial_{s'}\gsvs\H+\epsilon \|\frac{d}{ds'}(L_{s'}^{-1}{P_{s'}^c\partial_{s'} \gsvs}) \H \nonumber
\\&\phantom{aaaaaaaaaaaaaaaaaaaaaaaaaaaaaaaa}+\frac{1}{\epsilon}\|N(\gsvs,\varepsilon
L_{s'}^{-1}P_{s'}^c\partial_{s'}\gsvs+{\vec{\tilde{\phi}}}_{s'})\H \Big)\nonumber \\
&\lesssim \eps + \int_{0}^{s}ds' \ \left( \frac{1}{\eps} (\eps \jse^{-\frac{3}{2}} + \eps^{2}) + \eps  \right) \nonumber \\
& \lesssim \eps\, . \nonumber
\end{align}
 Note that the implicit multiplicative constant can be chosen to be uniform in $s\in [0,s_{0}]$.

 The case $L_0^{-1}P_0^c \partial_s \vec{\psi}_{E_0,s}|_{s=0}=0$ is easier: Estimate~(\ref{locest}) is then modified to
\begin{align} \label{locestprime}
\| {\vec{\tilde{\phi}}}_{s} \Hls \leq C_{s_{0},A} \left( \delta(0,s)\ \Ml +  \eps^2\right)\, .
\end{align}
 With (\ref{deltachoice}) one obtains
 \begin{align}
 M_{s_{0}}^{(\mathrm{l})}\lesssim \eps^{2},
 \end{align}
which is (\ref{locestrepl}). Estimating (\ref{initial}-\ref{nonlinear}) in $\|\cdot\H$ similarly as above yields
 \begin{align*}
\M \lesssim \eps \, ,
 \end{align*}
 with an implicit multiplicative constant being uniform in $s\in [0,s_{0}]$. This proves Proposition~\ref{gpshortestlem}.
\end{proof}

\subsection{Proof of Lemma~\ref{lemtriplestar}}
Next, we estimate \eqref{initial}-\eqref{nonlinear} in the space $H^{2,-\sigma}(\R^{3})$ term by term.

\paragraph{Local estimate for (\ref{initial}):}
If $L_0^{-1}P_0^c \partial_s \vec{\psi}_{E_0,s}|_{s=0}=0$ there is nothing to do. Otherwise, use the estimate 
\begin{align*}
\| \japx^{-\sigma}\japx^{\sigma}L^{-1}_{s}P_{s}^{c}\partial_{s}\gsv\|_{W^{2,1}}\lesssim \|\japx^{-\sigma}\H \|L_{s}^{-1}P_{s}^{c}\partial_{s}\gsv\|_{H^{2,\sigma}}\, ,
\end{align*}
and Theorem \eqref{cor1} to obtain
\begin{align*}
\|U_0(s,0)P^c_{H_0}{{\vec{\tilde{\phi}}}}_{0}\Hls \leq A  \jte^{-\frac{3}{2}} \eps  \|   L_{0}^{-1}P_{0}^{c}\partial_{s}\gsv|_{s=0}\|_{H^{2}\cap W^{2,1}}   \simeq \jte^{-\frac{3}{2}}    \eps  \, .
\end{align*}

\paragraph{Local estimate for \eqref{lin}:} To apply Theorem \ref{cor1} we need bounds for the $\| \cdot \|_{H^{2}\cap W^{2,1}}$-norms of each term:
\begin{lem} \label{est:locnorm} 
\begin{align} \label{firstestlin}
\|P_{H_{0}}^{c}(L_{s}+JH_{E_{s'}^{\epsilon},0})\dist\|_{H^{2}\cap W^{2,1}}& \lesssim \delta(0,s)\ \|\dist \Hls\, , \\ \label{secondlin}
\|P_{H_{0}}^{c}P_{s}^{d}J\dist \|_{H^{2}\cap W^{2,1}} &\lesssim \|\dist \Hls \, ,\\ \label{thirdlin}
\|P_{H_{0}}^{c}P_{s}^{d}\dot{P}_{s}^{d}\dist\|_{H^{2}\cap W^{2,1}}  & \lesssim \|\dist \Hls \, .
\end{align}
\end{lem}
A proof of this lemma is given in Section \ref{sec:D}. 

In what follows we estimate \eqref{lin}, and start with the case $L_0^{-1}P_0^c \partial_s \vec{\psi}_{E_0,s}|_{s=0}\neq 0$. 

Lemma \ref{est:locnorm}, Theorem \ref{cor1}, and Estimate \eqref{gammaest} yield, for $s \leq s_{0}$ and $s_{0},\eta,\eps \ll 1$,
\begin{align*}
\Big\|\int_0^sds'& \ U_0(s,s') \left(   \frac{1}{\epsilon}P^c_{H_0}(L_{s'}+JH_{E_{s'}^{\epsilon},0}){\vec{\tilde{\phi}}}
_{s'}+\frac{\dot{\gamma}^{\eps}_{s'}}{\eps}P_{H_{0}}^{c}P_{s'}^{d}J\dists-P^c_{H_0}P^{d}_{s}\dot{P}^d_s{\vec{\tilde{\phi}}}_{s'} \right)  \Big\|_{H^{2,-\sigma}}   \\
&\lesssim \int_{0}^{s}ds' \langle \frac{s-s'}{\eps} \rangle^{-\frac{3}{2}}\Big(\frac{1}{\eps} \|P_{H_{0}}^{c}(L_{s'}+JH_{E_{s'}^{\epsilon},0})\dists\|_{H^{2}\cap W^{2,1}}   + \eps \|P_{H_{0}}^{c}P_{s'}^{d}J\dists \|_{H^{2}\cap W^{2,1}}  \\
&\hspace{2in}+\|P_{H_{0}}^{c}P_{s'}^{d}\dot{P}_{s'}^{d}\dists\|_{H^{2}\cap W^{2,1}}   \Big) \\
&\lesssim \int_{0}^{s}ds' \langle \frac{s-s'}{\eps} \rangle^{-\frac{3}{2}}  \left(    ( \eps \langle \frac{s'}{\eps}\rangle^{-\frac{3}{2}} +\eps^{2})(\eps  \langle \frac{s'}{\eps}\rangle^{-\frac{3}{2}} +\eps^{2})^{-1}   \frac{\delta(0,s)}{\eps}\  \|\dists\Hls  + \|\dists\Hls \right) \\
&\lesssim {\delta}(0,s)\ (\eps \jte^{-\frac{3}{2}}+    \eps^{2}  ) M^{(\frac{3}{2},\mathrm{l})}_{s} + \eps^{2}\, ,
\end{align*}
with $$M^{(\frac{3}{2},\mathrm{l})}_{s}= \sup_{0\leq s' \leq s} (\eps \langle
\frac{s'}{ \eps }\rangle^{-\frac{3}{2}}+\eps^{2} )^{-1} \| \dists \Hls,$$ recall the constant $\delta(0,s)$ from \eqref{est:locality}. In the last inequality we applied \Bsl \ as well as the following key observations:
\begin{lem} \label{integral}
\begin{align*}
\int_0^sds' \jtse^{-\frac{3}{2}} &\lesssim \eps\, , \\
\int_0^sds' \jtse^{-\frac{3}{2}} \jse^{-\frac{3}{2}}  &\lesssim \eps \jte^{-\frac{3}{2}}\, .
\end{align*}
\end{lem}
\begin{proof}
The first estimate follows immediately after a change of variables $s'\rightarrow \frac{s'}{\epsilon}$. For the second we divide the integral region into two parts $[0,s/2]$ and $[s/2,s]$ to obtain
\begin{align*}
\int_0^sds' \jtse^{-\frac{3}{2}} \jse^{-\frac{3}{2}}  &= \int_0^{s/2}ds' \jtse^{-\frac{3}{2}} \jse^{-\frac{3}{2}} + \int_{s/2}^sds' \jtse^{-\frac{3}{2}} \jse^{-\frac{3}{2}} \\
&\lesssim  \langle \frac{s}{2\eps} \rangle^{-\frac {3}{ 2} } \left( \int_0^{s/2}ds'  \jse^{-\frac{3}{2}} +  \int_{s/2}^{s}ds'  \jtse^{-\frac{3}{2}} \right) \\
&\lesssim \eps \jte^{-\frac{3}{2}}\, .
\end{align*}
\end{proof}

Next we estimate \eqref{lin} for the simpler case $L_0^{-1}P_0^c \partial_s \vec{\psi}_{E_0,s}|_{s=0}=0$:
\begin{align*}
\Big\|\int_0^sds'& \ U_0(s,s') \left(   \frac{1}{\epsilon}P^c_{H_0}(L_{s'}+JH_{E_{s'}^{\epsilon},0}){\vec{\tilde{\phi}}}
_{s'}+\frac{\dot{\gamma}^{\eps}_{s'}}{\eps}P_{H_{0}}^{c}P_{s'}^{d}J\dists-P^c_{H_0}P^{d}_{s}\dot{P}^d_s{\vec{\tilde{\phi}}}_{s'} \right)  \Big\|_{H^{2,-\sigma}}   \\
&\lesssim \int_{0}^{s}ds' \langle \frac{s-s'}{\eps} \rangle^{-\frac{3}{2}}\Big(\frac{1}{\eps} \|P_{H_{0}}^{c}(L_{s'}+JH_{E_{s'}^{\epsilon},0})\dists\|_{H^{2}\cap W^{2,1}}   + \eps \|P_{H_{0}}^{c}P_{s'}^{d}J\dists \|_{H^{2}\cap W^{2,1}}  \\
&\hspace{2in}+\|P_{H_{0}}^{c}P_{s'}^{d}\dot{P}_{s'}^{d}\dists\|_{H^{2}\cap W^{2,1}}   \Big) \\
&\lesssim \int_{0}^{s}ds' \langle \frac{s-s'}{\eps} \rangle^{-\frac{3}{2}}  \left(      \frac{\delta(0,s)}{\eps}  \|\dists\Hls  + \|\dists\Hls \right) \\
&\lesssim {\delta(0,s)}\Ml + \eps^{2}\,.
\end{align*}

\paragraph{Local estimate for (\ref{inhomo}):}
\begin{align*}
\Big\|\int_0^sds'& \ U_0(s,s') \left(-\dot{\gamma}^{\eps}_{s'}P_{H_0}^cP_{s'}^cJL_{s'}^{-1}P_{s'}^c\partial_{s'}\gsvs-\epsilon P^c_{H_0}P_{s'}^c\frac{d}{ds'}(L_{s'}^{-1}{P_{s'}^c\partial_{s'} \gsvs}) \right) \Big\|  \\
&\lesssim \int_{0}^{s}ds' \ \jtse^{-\frac{3}{2}} \eps \left(\Big\|  L_{s'}^{-1}P_{s'}^c\partial_{s'}\gsvs\Big\|_{H^{2}\cap W^{2,1}}+\Big\| \frac{d}{ds'}(L_{s'}^{-1}{P_{s'}^c\partial_{s'} \gsvs} ) \Big\|_{H^{2}\cap W^{2,1}}\right) \\
&\lesssim \eps^2 .
\end{align*}
The first inequality results from Estimate~(\ref{gammaest}) and the fact that $\|P_{H_{0}}^{c}P_{s}^{c}\|_{H^{2}\cap W^{2,1}\rightarrow H^{2}\cap W^{2,1}}$ is uniformly bounded. The second inequality follows from Lemma \ref{integral} and \Bsl \ as well as $\|\cdot\|_{H^{2}\cap W^{2,1}} \lesssim \|\cdot \|_{H^{2,\sigma}}$. 

\paragraph{Local estimate for (\ref{nonlinear}):}
Instead of expanding $N(\gsvs,\varepsilon
L_{s'}^{-1}P_{s'}^c\partial_{s'}\gsvs+{\vec{\tilde{\phi}}}_{s'})$ it is more convenient to consider $N(\gsvs, \vec{\phi}_{s'})$, recall that in \eqref{finerdecomp}, $$\vec{\phi}_{s}=\epsilon L_{s}^{-1}P_{s}^c\partial_{s}\gsv+{\vec{\tilde{\phi}}}_{s}.$$ By Equation \eqref{eq_disp1} we may conclude that 
\begin{itemize}
\item[-] terms which are quadratic in $\phi_{s'}$ come with a factor of $\psi_{E^{\eps}_{s'},s'}$ which decays rapidly at spatial infinity. By $(\mathrm{B}_{\mathrm{l}})$
\begin{align*}
      \|\psi_{E_{s'}^{\eps},s'} \phi_{s'}^{2}\|_{H^{2}\cap W^{2,1}}   =   \|\japx^{2\sigma}\psi_{E_{s'}^{\eps},s'} \japx^{-2\sigma}\phi_{s'}^{2}\|_{H^{2}\cap W^{2,1}} \lesssim \| \dists \Hls^{2} \lesssim \eps^{2} \, ,
\end{align*}
\item[-] terms which are cubic in $\phi_{s'}$ are estimated by
\begin{align*}
\||\phi_{s'}|^{2}\phi_{s'} \|_{H^{2}\cap W^{2,1}} \lesssim \|\dists\H^{3}\lesssim  \eps^{2}\, .
\end{align*}
Here the bootstrap assumption~$(\mathrm{B}_{\mathrm{g}})$ for the global norm $\|\distt\H$ has been used.
\end{itemize}
Hence collect the estimates above to obtain
\begin{align*}
{\big{\|}}\int_0^sds' \ U_0(s,s') \frac{1}{\epsilon}P^c_{H_0}P_{s'}^cN(\gsvs,{\vec{{\phi}}}_{s'})    \Hls 
& \lesssim \int_{0}^{s}ds' \ \jtse^{-\frac{3}{2}} \frac{1}{\eps}\cdot \eps^{2} \lesssim \eps^{2}\, .
\end{align*}
This finishes the proof of Lemma~\ref{lemtriplestar}.\hfill$\square$

\section{Proof of Theorem~\ref{refor} for all $s\in [0,1]$}\label{sec:full}
So far we have established Theorem~\ref{refor} on the small interval $[0,s_{0}]$ only. Recall that $s_{0}$ does not depend on $\eps$ if $\eps \ll 1$. Next we extend the results to the interval $[0,1].$

\begin{prop} \label{gplemall} For $\eta \ll 1$ there exists a small time $\tau^{*}>0$ and constant $C_{\tau^{*}}$ with the following property: Whenever
{\begin{align*}
M_{{s}^{*}}^{\mathrm{(l)}} &\leq C_{{s}^{*}}\eps\left(\jte^{-\frac{3}{2}}+\eps\right)\\
M_{{s}^{*}}^{(\mathrm{g})}&\leq C_{{s}^{*}}\eps
\end{align*}}
hold for some ${s}^{*}\in [s_{0},1]$, then
{\begin{align}
M_{{s}^{*}+\tau^{*}}^{\mathrm{(l)}} &\leq C_{\tau^{*}}C_{{s}^{*}}\eps\left(\jte^{-\frac{3}{2}}+\eps\right)\label{finalimprovement1} \\
M_{{s}^{*}+\tau^{*}}^{(\mathrm{g})}&\leq C_{\tau^{*}}C_{{s}^{*}}\eps \label{finalimprovement2}
\end{align}}
for all $\eps \ll 1$.
\end{prop}

\begin{proof}[Proof of Theorem~\ref{refor}] Clearly the hypothesis of the proposition is satisfied at ${s}^{*}=s_{0}$. Estimate~(\ref{distestn1}) follows by iteration. Estimates~(\ref{distEest}) and (\ref{gammaquadest}) are proven by the same techniques as before. 
\comment{By Proposition~\ref{lem_wellpos} and Estimate~(\ref{ineq_locest}) there is a maximal $s_{0}<\tau \leq 1$ such that $(\mathrm{B}_{\mathrm{l}}$, $\mathrm{B}_{\mathrm{g}})$ resp. $(\mathrm{B}_{\mathrm{l}}'$, $\mathrm{B}_{\mathrm{g}})$ \ are satisfied with $s=\tau$ as long as $\eps \ll 1$. Actually, $\tau=1$ because if this were false we could -- by Proposition~\ref{gplemall} -- extend the validity of $(\mathrm{B}_{\mathrm{l}}$,~$\mathrm{B}_{\mathrm{g}})$ resp. $(\mathrm{B}_{\mathrm{l}}'$, $\mathrm{B}_{\mathrm{g}})$ \ to a slightly bigger interval in contradiction to the maximality. This proves Estimate~(\ref{distestn1}). Estimates (\ref{distEest}, \ref{gammaquadest}) are now seen to hold by the analysis in Subsection~\ref{contrmodpar}.}
\end{proof}

\begin{proof}[Proof of Proposition~\ref{gplemall}] 
The choice of $\tau^*$ will be in \eqref{choiceTau}, after considering all the factors determining it.

We reformulate the equation for $\dot{{\vec{\tilde\phi}}}_s$ similar to that in \eqref{phitildeq11}. The only difference is that we approximate $L_s$ by $-JH_{E_{s}^{\eps},s^*}$ with
\begin{align} \label{chapgp:refham1}
H_{E^{\eps}_{s},s^*}:=-\Delta+
V_{s^*}-E^{\eps}_{s}=H_{s^*}-E^{\eps}_{s}.
\end{align}
Hence the equation for $\dot{{\vec{\tilde\phi}}}_s$ becomes
\begin{align} \label{phitildeq111}
&\dot{{\vec{\tilde\phi}}}_s=-\frac{1}{\varepsilon}J(H_{E_{s}^{\eps},s^*}+\dot{\gamma}^{\eps}_s){\vec{\tilde\phi}}
_s   \\
& \phantom{aaa} +   \frac{1}{\epsilon}(L_s+JH_{E_{s}^{\eps},s^*}){\vec{\tilde\phi}}
_s    + \frac{\dot{\gamma}^{\eps}_s}{\eps}P_s^dJ\distt   -P_{s}^{d}\dot{P}^d_s{\vec{\tilde\phi}}_s \nonumber  \\
& \phantom{aaa}
-\dot{\gamma}^{\eps}_sP_s^cJL_s^{-1}P_s^c\partial_{s}\gsv 
-
\epsilon
P_s^c\frac{d}{ds}(L_s^{-1}{P_s^c\partial_{s}\gsv}) -\frac{1}{\epsilon}P_s^cN(\gsv,\varepsilon
L_s^{-1}P_s^c\partial_{s}\gsv+{\vec{\tilde\phi}}_s) \, . \nonumber 
\end{align}
Now we consider the initial condition, since the equation for $\dot{{\vec{\tilde\phi}}}_s$ in \eqref{phitildeq11} is started from time $s^*$, its initial condition ${\vec{\tilde\phi}}_{s^*}$ takes the form
\begin{align}
{\vec{\tilde\phi}}_{s^*}=&U_{{s^*}}(s^*,0){{\vec{\tilde{\phi}}}}_{0}+
\int_0^{s^*} \ ds' \ U_{{s^*}}  (s^*,s') \times \label{eq:exprTildPhi} \\ 
&\phantom{aaaaa} \times \Big(  
\frac{1}{\epsilon}(L_{s'}+JH_{E_{s'}^{\epsilon},s^*}){\vec{\tilde{\phi}}}
_{s'}+\frac{\dot{\gamma}^{\eps}_{s'}}{\eps}P_{s'}^dJ\dists-P^{d}_{s'}\dot{P}^d_{s'}{\vec{\tilde{\phi}}}_{s'}  \nonumber\\
&\phantom{aaaaa} -\dot{\gamma}^{\eps}_{s'}P_{s'}^cJL_{s'}^{-1}P_{s'}^c\partial_{s'}\gsvs  -\epsilon
P_{s'}^c\frac{d}{ds'}(L_{s'}^{-1}P_{s'}^c\partial_{s'}
\gsvs)\nonumber
\\
& \phantom{aaaaa}-\frac{1}{\epsilon}P_{s'}^cN(\gsvs,\varepsilon
L_{s'}^{-1}P_{s'}^c\partial_{s'}\gsvs+{\vec{\tilde{\phi}}}_{s'})\Big)\, .\nonumber
\end{align}
Then for any time $s\geq s^*$, apply Duhamel's principle on \eqref{phitildeq111} to obtain
\begin{align} 
P_{H_{s^*}}^c{{\vec{\tilde{\phi}}}}_s=&P_{H_{s^*}}^cU_{{s^*}}(s,s^*){{\vec{\tilde{\phi}}}}_{s^*}+\int_{s^*}^s \ ds' \ U_{{s^*}}  (s,s') \times \nonumber \\ 
&\phantom{aaaaa} \times \Big(  
\frac{1}{\epsilon}P_{H_{s^*}}^c(L_{s'}+JH_{E_{s'}^{\epsilon},s^*}){\vec{\tilde{\phi}}}
_{s'}+\frac{\dot{\gamma}^{\eps}_{s'}}{\eps}P_{H_{s^*}}^cP_{s'}^dJ\dists-P_{H_{s^*}}^cP^{d}_{s'}\dot{P}^d_{s'}{\vec{\tilde{\phi}}}_{s'}   \nonumber\\
&\phantom{aaaaa} -\dot{\gamma}^{\eps}_{s'}P_{H_{s^*}}^cP_{s'}^cJL_{s'}^{-1}P_{s'}^c\partial_{s'}\gsvs  -\epsilon
P_{H_{s^*}}^cP_{s'}^c\frac{d}{ds'}(L_{s'}^{-1}P_{s'}^c\partial_{s'}
\gsvs)
\nonumber\\
& \phantom{aaaaa}-\frac{1}{\epsilon}P_{H_{s^*}}^cP_{s'}^cN(\gsvs,\varepsilon
L_{s'}^{-1}P_{s'}^c\partial_{s'}\gsvs+{\vec{\tilde{\phi}}}_{s'})\Big)\, .
\nonumber
\end{align}

Plug the expression for ${{\vec{\tilde{\phi}}}}_{s^*}$ in \eqref{eq:exprTildPhi} and use the semigroup property $$U_{{s^*}}(s,s^*)U_{{s^*}}(s^*,0)=U_{{s^*}}(s,0)$$ to obtain
\begin{align}
P_{H_{s^*}}^c{{\vec{\tilde{\phi}}}}_s=&
U_{{s^*}}(s,0)P_{H_{s^*}}^c{{\vec{\tilde{\phi}}}}_{0}\label{1initial} \\ \label{1homo} \phantom{aa}&+
\left(\int_0^{s^{*}} +\int_{s^*}^s \right)ds' \ U_{{s^*}}  (s,s') \times \nonumber \\ 
&\phantom{aaaaa} \times \Big(  
\frac{1}{\epsilon}P_{H_{s^*}}^c(L_{s'}+JH_{E_{s'}^{\epsilon},s^*}){\vec{\tilde{\phi}}}
_{s'}+\frac{\dot{\gamma}^{\eps}_{s'}}{\eps}P_{H_{s^*}}^cP_{s'}^dJ\dists-P_{H_{s^*}}^cP^{d}_{s'}\dot{P}^d_{s'}{\vec{\tilde{\phi}}}_{s'}  \\ \label{1inhomo}
&\phantom{aaaaa} -\dot{\gamma}^{\eps}_{s'}P_{H_{s^*}}^cP_{s'}^cJL_{s'}^{-1}P_{s'}^c\partial_{s'}\gsvs  -\epsilon
P_{H_{s^*}}^cP_{s'}^c\frac{d}{ds'}(L_{s'}^{-1}P_{s'}^c\partial_{s'}
\gsvs)
\\&\label{1nonlin} \phantom{aaaaa}-\frac{1}{\epsilon}P_{H_{s^*}}^cP_{s'}^cN(\gsvs,\varepsilon
L_{s'}^{-1}P_{s'}^c\partial_{s'}\gsvs+{\vec{\tilde{\phi}}}_{s'})\Big)\, .\end{align}

In what follows we estimate \eqref{1initial}-\eqref{1nonlin}. As usual we assume $(\mathrm{B}_{\mathrm{l}}$, $\mathrm{B}_{\mathrm{g}})$ resp. $(\mathrm{B}_{\mathrm{l}}'$, $\mathrm{B}_{\mathrm{g}})$ hold in a time interval $[s^*,s^*+\tau]$ for some $\tau>0$. The existence of $\tau>0$ is guaranteed by the local wellposedness of the solution. Recall that the $(\mathrm{B}_{\mathrm{l}}$, $\mathrm{B}_{\mathrm{g}})$ resp. $(\mathrm{B}_{\mathrm{l}}'$, $\mathrm{B}_{\mathrm{g}})$ conditions are defined as
\begin{itemize}
\item[$(\mathrm{B}_{\mathrm{l}})$] \hspace{1in}
$M_{{s}}^{(\mathrm{l})} \leq 2A\eps \| L_0^{-1}P_0^c \partial_s \vec{\psi}_{E_{s},s} |_{s=0} \|_{{H^2} \cap W^{2,1}}+\epsilon $  ,
\item[$(\mathrm{B}_{\mathrm{g}})$] \hspace{1in}
$M_{{s}}^{(\mathrm{g})}\leq  \eps^{\frac 2 3}$\, ,
\end{itemize}
where $A>1$ is the constant in the dispersive Estimate \eqref{dispest}. Note that if $L_0^{-1}P_0^c \partial_s \vec{\psi}_{E_0,s}|_{s=0}=0$ then \Bsl $\ $ is replaced by
\begin{itemize}
\item[$(\mathrm{B}_{\mathrm{l}}')$] \hspace{1in}
$M_{{s}}^{(\mathrm{l})} \leq \eps$\, .
\end{itemize}

In estimating \eqref{1initial}-\eqref{1nonlin}, the decay estimates generated by $U_{{s^*}}(s,t)$ play a prominent role, as in the proof of Proposition \ref{gpshortestlem}. Here after analyzing as in \eqref{eq:relateHu}, we find that it suffices to study the operator $e^{-it H_{s^*}}$, which makes the results in Theorem \ref{cor1} applicable.

\paragraph{Local estimate for (\ref{1initial}):} It is easy to obtain
\begin{align*}
\|U_{s^*}(s,0) P^c_{H_{s^*}} {\vec{\tilde{\phi}}}_0 \Hls \leq  A  \jte^{-\frac{3}{2}} \eps  \|   L_{0}^{-1}P_{0}^{c}\partial_{s}\gsv|_{s=0}\|_{H^{2}\cap W^{2,1}}   \lesssim    \eps^{2}  \, .
\end{align*}
Here we used $s\geq s_{0}>0$.
\paragraph{Local estimate for (\ref{1homo}):}
The integrals of the first summand {$\frac{1}{\epsilon}P_{H_{s^{*}}}^{c}(L_{s'}+JH_{E_{s'}^{\epsilon},s^*}){\vec{\tilde{\phi}}}
_{s'}$} dominate the others for $\eps \ll 1$. By \eqref{ineq_locest} and Lemmata \ref{est:locnorm}, \ref{integral} 
\begin{align*}
\Big\| \frac{1}{\epsilon} \int_0^{s^{*}} ds' \ U_{{s^*}}(s,s')P_{H_{s^*}}^c  
(L_{s'}+JH_{E_{s'}^{\epsilon},s^*}){\vec{\tilde{\phi}}}
_{s'} \Big\|_{H^{2,-\sigma}} \leq &\frac{C_{s^{*}}A}{\eps} \int_0^{s^{*}} ds' \ \langle \frac{s-s'}{\eps} \rangle^{-\frac{3}{2}} (\eps \jse^{-\frac{3}{2}}+\eps^2) \\
\leq&C_{s^{*},A}\left( \eps \jte^{-\frac{3}{2}}+\eps^{2}\right)\\
\leq&  C_{s^{*},A}\eps^{2} \, ,
\end{align*}
if $L_0^{-1}P_0^c \partial_s \vec{\psi}_{E_0,s}|_{s=0}\neq0$. The same estimate holds if $L_0^{-1}P_0^c \partial_s \vec{\psi}_{E_0,s}|_{s=0}=0$.

Next, we use (\ref{firstestlin}) in Lemma \ref{est:locnorm} with $0$ replaced by $s_{*}$. There exists a constant $C_{1}=C_{1}(A)>0$, so that for all $\eps \ll 1$
\begin{align*}
\Big\| \frac{1}{\epsilon} \int_{s^*}^s ds' \ U_{{s^*}}(s,s')P_{H_{s^*}}^c  
(L_{s'}+JH_{E_{s'}^{\epsilon},s^*}){\vec{\tilde{\phi}}}
_{s'} \Big\|_{H^{2,-\sigma}} &\leq A \frac{\delta(s^*,s)}{\eps} \int_{s^*}^s ds' \ \langle \frac{s-s'}{\eps} \rangle^{-\frac{3}{2}} \|{\vec{\tilde{\phi}}}_{s'}\Hls \\& \leq C_{1} \delta(s^*,s) \sup_{s^*\leq s' \leq s} \|{\vec{\tilde{\phi}}}_{s'}\Hls \, .
\end{align*}
Recall the definition of $\delta(u,s)\in \mathbb{R}^{+}$ from \eqref{est:locality}.

\paragraph{Local estimate for (\ref{1inhomo}):}The integral of the second summand $\epsilon
P_{H_{s^{*}}}^{c}P_{s'}^c\frac{d}{ds'}(L_{s'}^{-1}P_{s'}^c\gsvs)$ dominates the other and
\begin{align*}
\eps \Big\|\int_0^s ds' \ U_{{s^*}}(s,s')P_{H_{s^*}}^c
P_{s'}^c\frac{d}{ds'}(L_{s'}^{-1}P_{s'}^c\partial_{s'}
\gsvs) \Big\|_{H^{2,-\sigma}} \lesssim C_{s^{*}} \eps \int_0^s ds' \ \langle \frac{s-s'}{\eps} \rangle^{-\frac{3}{2}} \lesssim \eps^2\, .
\end{align*}

\paragraph{Local estimate for (\ref{1nonlin}):} By the same reasoning as in the local estimate for (\ref{nonlinear}) we have
\begin{align*}
\Big\| \int_0^s ds' \ U_{s^*}(s,s')&\left(-\frac{1}{\epsilon}P^{c}_{H_{s^{*}}}P_{s'}^cN(\gsvs,\varepsilon
L_s^{-1}P_s^c\partial_{s}\gsvs+{\vec{\tilde{\phi}}}_s)\right)\Big\|_{H^{2,-\sigma}}  \lesssim C_{s^{*}}\eps^2\, .
\end{align*}

Collect the estimates above to conclude that, for $s \in [s^*,\ \tau]$ and for all $\eps \ll 1$,
\begin{align*}
\| \distt \Hls \leq C_{2}  \eps^{2}+C_{1}\delta(s^*,s) \sup_{s^{*}\leq s' \leq s^{*}+\tau^{*}} \| \dists \Hls\, .
\end{align*}
Next we choose a (possibly small) $\tau^*$ such that 
\begin{align}
C_{1}\delta(s^*,s^*+\tau^*)\leq \frac{1}{2}.\label{choiceTau}
\end{align}
Then it follows that for $\tau^{**}=min\{s^*+\tau^*,\ \tau\}$
 \begin{align} \label{finallocest}
\sup_{s^{*}\leq s' \leq \tau^{**}} \|\dists \Hls \leq 2C_{2} \eps^2\, .
\end{align}

Next we estimate $\| \dists \H$.
We estimate \eqref{1initial}-\eqref{1nonlin} in the norm $\| \cdot \H$ as in the proof of Proposition \ref{gpshortestlem}, with the help of \eqref{finallocest}, to find
\begin{align}\label{finalglobest}
 \sup_{s^{*}\leq s' \leq s^{*}+\tau^{*}} \|\dists \H \lesssim C_{s^{*}} \eps\, .
\end{align}

It is not hard to see that \eqref{finallocest}, \eqref{finalglobest} imply the desired \eqref{finalimprovement1}, \eqref{finalimprovement2}, if we prove $\tau^{**}=s^*+\tau^*$. 

This is indeed true. By the local wellposedness of the solution, the bootstrap assumptions $(\mathrm{B}_{\mathrm{l}}$, $\mathrm{B}_{\mathrm{g}})$ resp. $(\mathrm{B}_{\mathrm{l}}'$, $\mathrm{B}_{\mathrm{g}})$ hold for a maximal subinterval $[s^{*}, s^{*}+\tau]\subset[s^{*}, s^{*}+\tau^{*}]$, where $\tau$ a priori depends on $\eps$. We claim that $\tau=\tau^{*}$: If we assume $\tau<\tau^{*}$, then the estimates \eqref{finallocest}, \eqref{finalglobest}, which are better than $(\mathrm{B}_{\mathrm{l}}$, $\mathrm{B}_{\mathrm{g}})$ resp. $(\mathrm{B}_{\mathrm{l}}'$, $\mathrm{B}_{\mathrm{g}})$, still hold in $[0,\tau]$. But the local wellposedness implies that the bootstrap assumptions also hold on a bigger interval (for $\eps \ll 1$). This contradicts the maximality of $\tau$.
\end{proof}

\appendix


\section{Proof of Theorem \ref{cor1} }\label{sec:B}

The estimate \eqref{goldbergest} can be found in \cite{Goldberg:2006wz}. 

The estimate \eqref{propest} follows from
\begin{align}
\|\e^{-\i\frac{t}{\eps}H_{\tau}}P_{H_{\tau}}^{c}\chi \H &\simeq \|\e^{-\i\frac{t}{\eps}H_{\tau}}P_{H_{\tau}}^{c}\chi \|_2 + \|H_\tau \e^{-\i\frac{t}{\eps}H_{\tau}}P_{H_{\tau}}^{c}\chi \|_2 \nonumber  \\
& \simeq \|\e^{-\i\frac{t}{\eps}H_{\tau}}P_{H_{\tau}}^{c}\chi \|_2 + \|\e^{-\i\frac{t}{\eps}H_{\tau}} H_\tau P_{H_{\tau}}^{c}\chi \|_2 \nonumber \\ 
&\simeq \|P_{H_{\tau}}^{c}\chi \|_2+\| H_\tau P_{H_{\tau}}^{c}\chi \|_2 \nonumber \\
&\simeq \|P_{H_{\tau}}^{c}\chi \H \nonumber \\
& \simeq \|P_{H_{\tau}}^{c}\chi\|_2+\|P_{H_{\tau}}^{c}H_{\tau}\chi\|_{2} \lesssim \|\chi\H \label{estunitary} \, .
\end{align}
By inspection we see that all multiplicative constants can be chosen to be independent of $\tau$ due to $(\mathrm{H}_{\mathrm{d}})$.

Next we prove \eqref{dispest}.
By $ \|\japx^{-\sigma} \chi \|_2 \leq \| \japx^{-\sigma}\|_2 \| \chi \|_\infty \lesssim \|\chi\|_{\infty}$ for any $\sigma >2$ (recall $(\mathrm{H}_{\mathrm{d}})$) we obtain
\begin{align}
\| \e^{-\i\frac{t}{\eps}H_{\tau}}P_{H_\tau}^c \chi \Hls &\simeq \| \japx^{-\sigma}\e^{-\i\frac{t}{\eps}H_{\tau}}P_{H_\tau}^c \chi \|_2 + \| \japx^{-\sigma} H_\tau \e^{-\i\frac{t}{\eps}H_{\tau}}P_{H_\tau}^c \chi \|_2 \nonumber \\
&\simeq \| \japx^{-\sigma} \e^{-\i\frac{t}{\eps}H_{\tau}} P_{H_\tau}^c \chi \|_2 + \| \japx^{-\sigma}  \e^{-\i\frac{t}{\eps}H_{\tau}}P_{H_\tau}^c H_\tau \chi \|_2 \nonumber \\ 
& \lesssim \| \e^{-\i\frac{t}{\eps}H_{\tau}} P_{H_\tau}^c \chi \|_\infty + \|   \e^{-\i\frac{t}{\eps}H_{\tau}}P_{H_\tau}^c H_\tau \chi \|_\infty \nonumber \\
&\lesssim \left| \frac{t}{\eps} \right|^{-\frac{3}{2}} \left( \| \chi \|_1 + \| H_\tau \chi \|_1 \right) \nonumber \\
&\lesssim  \left| \frac{t}{\eps}\right|^{-\frac{3}{2}} \| \chi \|_{W^{2,1}}\, . \label{estdispersive}
\end{align}
This together with \eqref{estunitary} yields the desired estimate \eqref{dispest}.

The uniformity of the constant $A$ follows from compactness of $[0,1]$ and the following lemma.

\begin{lem}
Consider $H_{0}=-\Delta+V_{0}$ with $V_{0}\in H^{2,\sigma}(\R^{3})$ admitting no zero energy resonance, thus
\begin{align*}
\|\e^{-\i H_{0}t}P_{H_{0}}^{c}\phi\Hls \leq C_{0} \langle t \rangle ^{-\frac{3}{2}} \|\phi\|_{H^{2}\cap W^{2,1}}.
\end{align*}
Then for $H=-\Delta+V$, $V\in H^{2,\sigma}(\R^{3})$, $\|V-V_{0}\|_{H^{2,\sigma}}$ sufficiently small, it holds that 
\begin{align*}
\|\e^{-\i Ht}P_{H}^{c}\phi\Hls \leq C \langle t \rangle ^{-\frac{3}{2}} \|\phi\|_{H^{2}\cap W^{2,1}},
\end{align*}
where $C$ can be chosen such that $C\to C_{0}$ as $\|V-V_{0}\|_{H^{2,\sigma}}\to 0$.
\end{lem}
\begin{proof} To simplify the notation, $\delta>0$ will denote a generic quantity which tends to zero as \mbox{$\|V-V_{0}\|_{H^{2,\sigma}}\to 0$}. By Duhamel's formula
\begin{align}\label{dispersiveduhamel}
e^{-\i Ht}P_{H}^{c}\phi=e^{-\i P_{H}^{c}HP_{H}^{c}t}P_{H}^{c}\phi=&e^{-\i P_{H}^{c}H_{0}P_{H}^{c}t}P_{H}^{c}\phi \nonumber \\&\hspace{.1in}-\i\int_{0}^{t}ds \ e^{-\i P_{H}^{c}H_{0}P_{H}^{c}(t-s)}P_{H}^{c}(V-V_{0})P_{H}^{c}e^{-\i P_{H}^{c}HP_{H}^{c}s}P_{H}^{c}\phi \, .
\end{align}
{\itshape{Claim:}} $\|\e^{-\i P^{c}_{H}H_{0}P^{c}_{H}t}P^{c}_{H}\phi\Hls \leq \langle t \rangle^{-\frac{3}{2}}(C_{0}+\delta)\|\phi\|_{H^{2}\cap W^{2,1}}$.
\vspace{.1in}

\noindent We first show that the claim implies the lemma. With Lemma~\ref{lem_norm} we have
\begin{align*}
\|P_{H}^{c}(V-V_{0})\langle x \rangle^{\sigma}\langle x \rangle^{-\sigma}P_{H}^{c}e^{-\i P_{H}^{c}HP_{H}^{c}s}P_{H}^{c}\phi \|_{H^{2}\cap W^{2,1}} \leq \delta \|e^{-\i P_{H}^{c}HP_{H}^{c}s}P_{H}^{c}\phi \Hls
\end{align*}
 and thus estimating (\ref{dispersiveduhamel}) for all $t \leq t^{*}$ we obtain
\begin{align*}
\langle t \rangle ^{\frac{3}{2}} \|\e^{-\i Ht}P_{H}^{c}\phi\Hls \leq (C_{0}+\delta) \|\phi\|_{H^{2}\cap W^{2,1}}+ \delta\langle t \rangle^{\frac{3}{2}}\int_{0}^{t^{*}} &ds (C_{0}+\delta)\langle t - s \rangle^{-\frac{3}{2}}\langle s \rangle^{-\frac{3}{2}}  \times \\
&\times \sup_{s\leq t^{*}}\langle s \rangle^{\frac{3}{2}}\|e^{-\i P_{H}^{c}HP_{H}^{c}s}P_{H}^{c}\phi \Hls \, .
\end{align*}
The lemma now follows from 
\begin{align*}
\int_{0}^{\infty} ds \langle t-s \rangle^{-\frac{3}{2}}\langle s \rangle^{-\frac{3}{2}} \leq D \langle t \rangle^{-\frac{3}{2}}
\end{align*}
for a numerical constant $D$. This is proved as in Lemma~\ref{integral}.

 To prove the claim we define $\phi_{t}:=\e^{-\i P^{c}_{H}H_{0}P^{c}_{H}t}P^{c}_{H}\phi$ and apply Duhamel's formula again,
\begin{align*}
\phi_{t} &= \e^{-\i H_{0}t}P_{H}^{c}\phi+\i \int_{0}^{t}ds \ \e^{-\i H_{0}(t-s)}P_{H}^{d}H_{0}P_{H}^{c}\phi_{s} 
\end{align*}
It holds that
\begin{align*}
\|\phi_{t}\Hls=\|P_{H}^{c}\phi_{t}\Hls\leq& \|(P_{H_{0}}^{c}-P_{H}^{c})\phi_{t}\Hls + \|P_{H_{0}}^{c}\phi_{t}\Hls \\
 \leq& \delta \|\phi_{t}\Hls + \|P_{H_{0}}^{c}\phi_{t}\Hls\, ,
\end{align*}
and the same is true if $H^{2,-\sigma}(\R^{3})$ is replaced by $H^{2}(\R^{3})\cap W^{2,1}(\R^{3})$. The second inequality can be proved with the Riesz formula and certain resolvent estimates. It follows that it is sufficient to estimate
\begin{align*}
P_{H_{0}}^{c}\phi_{t}&= \e^{-\i H_{0}t}P_{H_{0}}^{c}P_{H}^{c}\phi+\i \int_{0}^{t}ds \ \e^{-\i H_{0}(t-s)}P_{H_{0}}^{c}P_{H}^{d}H_{0}P_{H}^{c}\phi_{s} \\
&=\e^{-\i H_{0}t}P_{H_{0}}^{c}\phi+\e^{-\i H_{0}t}P_{H_{0}}^{c}(P_{H}^{c}-P_{H_{0}}^{c}) \phi+\i \int_{0}^{t}ds \ \e^{-\i H_{0}(t-s)}P_{H_{0}}^{c}P_{H}^{d}(V_{0}-V)P_{H}^{c}\phi_{s} \, .
\end{align*}
Hence for all $t\leq t^{*}$
\begin{align*}
\langle t \rangle^{\frac{3}{2}}\|\phi_{t}\Hls &\leq (1+\delta)\langle t \rangle^{\frac{3}{2}}\|P_{H_{0}}^{c}\phi_{t}\Hls \\
& \leq (C_{0}+\delta) \|\phi \|_{H^{2}\cap W^{2,1}} \\
& \hspace{.5in} +(1+\delta) \langle t \rangle^{\frac{3}{2}}\int_{0}^{t}ds \langle t-s \rangle^{-\frac{3}{2}}\langle s \rangle^{-\frac{3}{2}}\langle s \rangle^{\frac{3}{2}}\|(V_{0}-V)\langle x \rangle^{\sigma}\langle x \rangle^{-\sigma}P_{H}^{c}\phi_{s}\|_{H^{2}\cap W^{2,1}} \\
&\leq (C_{0}+\delta) \|\phi \|_{H^{2}\cap W^{2,1}} + \delta D \sup_{s\leq t^{*}} \langle s \rangle^{\frac{3}{2}}\|\phi_{s}\Hls
\end{align*}
This proves the claim.
\end{proof}


\comment{
Ultimately, we show that for every $\eps>0$ there is a $\delta>0$ such that if
\begin{align*}
\|\e^{-\i\frac{t}{\eps}H_{\tau}}P^c_{H_\tau}\chi\Hls & \leq A_{\tau} \langle \frac{t}{\eps} \rangle^{-\frac 3 2}\|\chi\|_{{H^2}\cap W^{2,1}}\,
\end{align*}
then $A_{\tau+\lambda}$ can be chosen such that $|A_{\tau+\lambda}-A_{\tau}|<\eps$ if $|\lambda|<\delta$. The claim for $A$ then follows by compactness of $[0,1]$.
Now
\begin{align*}
 \e^{-\i\frac{t}{\eps}H_{\tau+\lambda}}P_{H_{\tau+\lambda}}^c \chi
\end{align*}
\tt{to be done!!} Duhamel argument!}

\section{Proof of Lemma \ref{normequiv}}\label{sec:C}

\begin{proof}
Obviously, $\|P^c_{H_0}\distt \|_{2} \leq \|\distt \|_{2}$ and $\|H_0 P^c_{H_0}\distt \|_{2}=\|P^c_{H_0}H_0\distt \|_{2} \leq \|H_0 \distt \|_2$ imply
\begin{align*}
\|  P^c_{H_0}\distt\|_{2}+\| P^c_{H_0}H_0\distt \|_{2}    \simeq  \|P^c_{H_0}\distt \|_{H^2} \lesssim \| \distt \|_{H^2}\, .
\end{align*}
By the definition of $P_{H_0}^{d}$ and that the eigenvectors of $-\Delta+V_0$ decay rapidly at $|x|=\infty,$ we have
\begin{align*}
\|\japx^{-\sigma} P^d_{H_0}\distt\|_2 \lesssim \| \japx^{-\sigma} \distt\|_2 \, ,
\end{align*}
and similarly, 
\begin{align*}
\| \japx^{-\sigma} H_{0} P_{H_0}^d \distt \|_2 =  \| \japx^{-\sigma}  P_{H_0}^d H_{0} \distt \|_2 \lesssim \|\dist \Hls \, ,
\end{align*}
whence $\|P_{H_{0}}^{c}\distt \Hls \lesssim \|\distt \Hls$. To show the converse inequalities note that for both norms
\begin{align*}
\|\distt \| \leq \| P_{H_{0}}^{c}\distt\| + \|\left( P_{s}-P_{H_{0}}^{c} \right)\distt\|\, .
\end{align*}
Therefore it suffices to show that $\|P_{s}-P_{H_{0}}^{c}\|_{H^{2}\to H^{2}}$ resp. $\|P_{s}-P_{H_{0}}^{c}\|_{H^{2,-\sigma}\to H^{2,-\sigma}}$ can be made arbitrarily small by choosing $s_{0}, \eta, \eps$ suitably. By the second resolvent formula we obtain
\begin{align*}
\|P_{s}-P_{H_{0}}^{c}\|\lesssim \oint_{\Gamma}dz \ \| (JL_{s}-z)^{-1}\| \|JL_{s}-H_{0}\| \| (H_{0}-z)^{-1} \|     \, ,
\end{align*}
$\Gamma$ as above. Observe that for arbitrary $\delta>0$ we can achieve
\begin{align*}
\|JL_{s}-H_{0}\| \lesssim \delta
\end{align*}
in both operator norms as a consequence of (\ref{chapgp:diffexpre}) and (\ref{est:locality}). Furthermore the norms of $(L_{s}-z)^{-1}$ and $(H_{0}-z)^{-1}$ are both uniformly bounded on $H^{2,-\sigma}(\R^{3})$ (the case $H^{2}(\R^{3})$ is easier). This concludes the proof.
\end{proof}

\section{Proof of Lemma \ref{est:locnorm}}\label{sec:D}

\begin{proof}
By Lemma \ref{normequiv} we have $$\|P_{H_{0}}^{c}(L_{s}+JH_{0})\dist\H \lesssim \|(L_{s}+JH_{0})\dist \H,$$ and $$\|P_{H_{0}}^{c}(L_{s}+JH_{0})\dist\|_{W^{2,1}} \lesssim \|(L_{s}+JH_{0})\dist \|_{W^{2,1}}.$$ Then, by means of Estimate~(\ref{est:locality})
\begin{align*}
\|(L_{s}+JH_{0})\japx^{\sigma}\japx^{-\sigma}\dist\H  \lesssim \delta \|\dist \Hls \, , \\
\|(L_{s}+JH_{0})\japx^{\sigma}\japx^{-\sigma}\dist\|_{W^{2,1}}  \lesssim \delta \|\dist \Hls\, .
\end{align*}
The proof of (\ref{secondlin}) is follows from similar arguments using the explicit expression for $P_{s}^{d}$, Equation~(\ref{eq:proj}). Ultimately, to prove (\ref{thirdlin}) we note that since \newline\mbox{$\frac{d}{dt}\gs , \frac{d}{dt}\partial_{E}\gs  \in H^{2,\sigma}(\R^{3})$} it holds that
\begin{align*}
\|P_{s}^{d}\dot{P}_{s}^{d}\dist\|_{H^{2}\cap W^{2,1}} \lesssim \|\dot{P}_{s}^{d}\dist\Hls \lesssim  \|\dist \Hls \, .
\end{align*}
\end{proof}

\def\cprime{$'$} \def\cprime{$'$} \def\cprime{$'$} \def\cprime{$'$}
  \def\cprime{$'$} \def\cprime{$'$}


\end{document}